\documentclass[a4paper]{amsart}

\usepackage{cite}
\usepackage{amsmath}
\usepackage{amssymb}
\usepackage{amsthm}   
\usepackage{empheq}   
\usepackage{graphicx}
\usepackage{color}
\usepackage{appendix}
\usepackage{hyperref}

\usepackage{amssymb}
\usepackage{graphicx}

\usepackage{hyperref}
\usepackage{breakurl}

\usepackage{times}
\usepackage{microtype}
\usepackage{amsmath}
\usepackage{tikz}
\usepackage{wasysym}
\usepackage{centernot}
\usepackage{mathrsfs}
\usepackage{cite}

\newtheorem{theorem}{Theorem}[section]
\newtheorem{lemma}[theorem]{Lemma}
\newtheorem{proposition}[theorem]{Proposition}
\newtheorem{corollary}[theorem]{Corollary}

\theoremstyle{definition}
\newtheorem{definition}[theorem]{Definition}
\newtheorem{example}[theorem]{Example}

\theoremstyle{remark}
\newtheorem{remark}[theorem]{Remark}

\numberwithin{equation}{section}

\def\EE{{\mathscr E}}
\def\FF{{\mathscr F}}

\def\rc{{\mathrm{c}}}
\def\rd{{\mathrm{d}}}

\def\rr{{\mathrm{r}}}
\def\rl{{\mathrm{l}}}
\def\rp{{\mathrm{p}}}
\def\rn{{\mathrm{n}}}
\def\fm{{\mathfrak{m}}}

\def\tI{{\mathtt{I}}}
\def\fs{{\mathfrak{s}}}
\def\fI{{\mathfrak{I}}}
\def\fa{{\mathfrak{a}}}
\def\fb{{\mathfrak{b}}}
\def\bR{{\mathbb{R}}}
\def\bZ{{\mathbb{Z}}}
\def\bN{{\mathbb{N}}}
\def\bQ{{\mathbb{Q}}}

\def\fJ{{\mathfrak{J}}}

\def\fe{{\mathfrak{e}}}

\def\ft{{\mathfrak{t}}}
\def\tJ{{\mathtt{J}}}

\def\fc{{\mathfrak{c}}}
\def\fd{{\mathfrak{d}}}

\def\bfP{{\mathbf{P}}}

\begin{document}

\title[On general skew Brownian motions]{On general skew Brownian motions}

\author{ Liping Li}
\address{RCSDS, HCMS, Academy of Mathematics and Systems Science, Chinese Academy of Sciences, Beijing 100190, China.}
\email{liliping@amss.ac.cn}
\thanks{The first named author is partially supported by NSFC (No. 11688101 and 11801546) and Key Laboratory of Random Complex Structures and Data Science, Academy of Mathematics and Systems Science, Chinese Academy of  Sciences (No. 2008DP173182).}






\begin{abstract}
The aim of this paper is two-fold. On one hand, we will study the distorted Brownian motion on $\bR$, i.e. the diffusion process $X$ associated with a regular and strongly local Dirichlet form obtained by the closure of $\EE(f,g)=\frac{1}{2}\int_\bR f'(x)g'(x)\rho(x)dx$ for $f,g\in C_c^\infty(\bR)$ on $L^2(\bR, \fm)$, where $\fm(dx)=\rho(x)dx$ and $\rho$ is a certain positive function. After figuring out the irreducible decomposition of $X$, we will present a characterization of that $X$ becomes a semi-martingale by virtue of so-called Fukushima's decomposition. Meanwhile, it is also called a general skew Brownian motion, which turns out to be a weak solution to the stochastic differential equation with certain $\mu$:
\begin{equation}\label{EQYTW}
	dY_t=dW_t+\int_\bR\mu(dz) dL^z_t(Y),
\end{equation}
where $(W_t)_{t\geq 0}$ is a standard Brownian motion and $(L^z_t(Y))_{t\geq 0}$ is the symmetric semi-martingale local time of the unknown semi-martingale $Y$ at $z$. On the other hand, the stochastic differential equation \eqref{EQYTW} will be considered further. The main purpose is to find the conditions on $\mu$ equivalent to that there exist general skew Brownian motions being weak solution to \eqref{EQYTW}. Moreover, the irreducibility and the equivalence in distribution of expected general skew Brownian motions will be characterized. Finally, several special cases will be paid particular attention to and we will prove or disprove the pathwise uniqueness for \eqref{EQYTW}.
\end{abstract}

\maketitle

\tableofcontents

\section{Introduction}

Consider the stochastic differential equation (SDE in abbreviation) of an unknown semi-martingale $(Y_t)_{t\geq 0}$:
\begin{equation}\label{EQ1YTW}
\left\lbrace
\begin{aligned}
	&dY_t=dW_t+\int_\bR \mu(dz) dL^z_t(Y),\\
	&Y_0=x \in \mathbb{R},
\end{aligned}\right.
\end{equation}
where $(W_t)_{t\geq 0}$ is a standard Brownian motion, $\mu=\mu^+-\mu^-$ is the difference of two positive measures on $\bR$ with $\mu^+\perp \mu^-$ and $(L^z_t(Y))_{t\geq 0}$ is the symmetric semi-martingale local time of $Y$ at $z$. Precisely speaking, 
\begin{equation}\label{EQ1LTY}
	L^z_t(Y)=\lim_{\varepsilon\downarrow 0} \frac{1}{2\varepsilon}\int_0^t 1_{(z-\varepsilon, z+\varepsilon)}(Y_s)d\langle Y\rangle_s,\quad \text{a.s.,} 
\end{equation}
where $\langle Y\rangle$ is the quadratic variation process of $Y$ (see \cite[Chapter VI \S1]{RY99}). Note that $t\rightarrow L^z_t(Y)$ is a.s. increasing and $dL^z_t(Y)$ in \eqref{EQ1YTW} means the differential of $L^z_t(Y)$ in $t$. A weak solution to \eqref{EQ1YTW} is a pair $(Y,W)$ on a certain probability space, where $Y$ is a continuous semi-martingale and $W$ is a standard Brownian motion, such that for all $t\geq 0$, $z\mapsto L^z_t(Y)$ is a.s. $|\mu|$-integral where $|\mu|:=\mu^++\mu^-$ and
\[
	Y_t-x=W_t+\int_\bR L^z_t(Y)\mu^+(dz)-\int_\bR L^z_t(Y)\mu^-(dz).
\]
We say the pathwise uniqueness holds for \eqref{EQ1YTW} if two weak solutions $(Y^1,W)$ and $(Y^2,W)$ with the same Brownian motion $W$ must coincide, i.e. $Y^1=Y^2$. The well-posedness of \eqref{EQ1YTW} means the existence of its weak solutions and that the pathwise uniqueness holds. 

A motivated example is $\mu=(2\alpha-1)\delta_0$ for a constant $\alpha\in (0,1)$. With this $\mu$, \eqref{EQ1YTW} is well posed and its solution is the so-called $\alpha$-skew Brownian motion (see \cite{HS81}). This process behaves like a Brownian motion except for the sign of each excursion is chosen by using an independent Bernoulli random variable of the parameter $\alpha$. Clearly, the symmetric case $\alpha=1/2$ coincides with Brownian motion. For $\alpha\neq 1/2$, the non-martingale part in \eqref{EQ1YTW} makes sense and the support $\{0\}$ of $(2\alpha-1)\delta_0$ is crucial, since $\alpha$-skew Brownian motion is ``skew" only at $0$. 
Generally in a celebrated work \cite{LG84}, Le Gall proves that the following assumption on $\mu$ is sufficient for the well-posedness of \eqref{EQ1YTW}: $|\mu|$ is finite and $|\mu(\{z\})|<1$ for any $z\in \bR$. Then Bass and Chen \cite{BC05} extend this result to the cases allowing $|\mu(\{z\})|=1$. Special situations are paid particular attention to by other researchers. For example, Ramirez \cite{Ra11} considers the case $\mu=\sum_{p\in \bZ}(2\alpha_p-1)\delta_{z_p}$, where $\alpha_p\in (0,1)$ and $\{z_p: p\in \bZ\}\subset \bR$ has no accumulation points. The associated diffusion as the unique weak solution to \eqref{EQ1YTW} is called a multi-skewed Brownian motion therein, and its infinitesimal generator and the conditions for its recurrence and positive recurrence are explored. In \cite{ORT15}, the set $\{z_p:p\in \bZ\}$ is replaced by another one with exactly one accumulation point, i.e. $\{z_p:p\in \bZ\}=\{0, l_k, r_k: k\in \bZ\}$ and 
\[
	\mu=\sum_{k\in \bZ}\left((2\alpha^-_k-1)\delta_{l_k}+(2\alpha^+_k-1)\delta_{r_k}\right)+(2\alpha_0-1)\delta_0,
	\]
where $\alpha^-_k,\alpha^+_k, \alpha_0\in (0,1)$, and $(l_k)_{k\in \mathbb{Z}}$ and $(r_k)_{k\in \mathbb{Z}}$ are two sequences of real numbers such that 
\begin{equation}\label{EQ1LKL}
\begin{aligned}
&l_k<l_{k+1}<0<r_k<r_{k+1}, \quad (\forall k\in \mathbb{Z}), \\
&\lim_{k\rightarrow \infty}l_k=0=\lim_{k\rightarrow -\infty}r_k,\quad \lim_{k\rightarrow -\infty}l_k=-\infty, \lim_{k\rightarrow \infty}r_k=\infty. 
\end{aligned}\end{equation}
It is proved that under a local version of Le Gall's condition: 
\[
	\sum_{k=1}^\infty|2\alpha^-_k-1|+\sum_{k=-\infty}^1|2\alpha^+_k-1|<\infty,
\] 
the SDE \eqref{EQ1YTW} is well posed. Then its unique solution is called a countably skewed Brownian motion in \cite{ORT15}. The properties of countably skewed Brownian motion like the non-explosion, recurrence and positive recurrence are further studied in \cite{ORT15}. Some other papers concerning \eqref{EQ1YTW} are \cite{Os82, Ta87} and the references therein.

In several works mentioned above such as \cite{Ra11, ORT15, Os82}, the theory of Dirichlet forms has been applied with some success to construct weak solutions to \eqref{EQ1YTW}. Let us use a few lines to explain some details. A Dirichlet form is a symmetric Markovian closed form on an $L^2(E,\fm)$ space, where $E$ is a nice topological space and $\fm$ is a fully supported Radon measure on it. Dirichlet forms are closely linked with Markov processes because of their Markovian property. Due to a series of important works by Fukushima, Silverstein in 1970's and Albeverio, Ma and R\"ockner in 1990's, it is now well known that a regular or quasi-regular Dirichlet form is always associated with a symmetric Markov process. We refer the notions and terminologies in the theory of Dirichlet forms to \cite{CF12} and \cite{FOT11}. In our case,  let $\fm(dx)=\rho(x)dx$ be a fully supported positive Radon measure on $\mathbb{R}$. This is meant to assume 
\begin{equation}\label{EQ2RLR}
\rho\in L^1_\mathrm{loc}(\mathbb{R}) \text{ and } \int_U \rho(x)dx>0
\end{equation}
for any non-empty open set $U$. Consider a quadratic form on $L^2(\mathbb{R},\fm)$:
\begin{equation}\label{EQ2EFG}
\mathscr{E}(f,g):=\frac{1}{2}\int_\mathbb{R}f'(x)g'(x)\rho(x)dx,\quad f,g\in C_c^\infty(\mathbb{R}). 
\end{equation} 
We need to point out except for \eqref{EQ2RLR}, another condition is necessary for the closability of \eqref{EQ2EFG}. That is, $\rho$ should vanish a.e. on its singular set (see \cite{H75} and \cite[Theorem~3.1.6]{FOT11})
\[
S(\rho):=\left\{x\in \mathbb{R}: \text{for any }\varepsilon>0, \int_{x-\varepsilon}^{x+\varepsilon}\frac{1}{\rho(y)}dy=\infty \right\}. 
\]
Under these two conditions, denote the closure of \eqref{EQ2EFG} by $(\EE,\FF)$. Clearly, it is a regular and strongly local Dirichlet form on $L^2(\mathbb{R}, \fm)$ and thus induces a diffusion process $X=\{(X_t)_{t\geq 0}, (\mathbf{P}_x)_{x\in \bR}\}$ on $\mathbb{R}$, which is usually called a \emph{distorted Brownian motion} (see \cite{Al77}). The reason that \eqref{EQ2EFG} interests us is that $X$ is possibly a semi-martingale and its martingale part is equivalent to a Brownian motion. For example, if $\rho$ is good enough (such as a continuously differential function with $1/c<\rho<c$ for some constant $c>1$), then $X$ is a semi-martingale under $\mathbf{P}_x$ for any $x\in \bR$ and
\begin{equation}\label{EQ1XTXB}
	X_t-x=B_t+\int_0^t\frac{\rho'}{2\rho}(X_s)ds= B_t+\int_{z\in \bR}L^z_t(X)\frac{\rho'(z)}{2\rho(z)}dz,
\end{equation}
where $B_t$ is a certain Brownian motion and $L^z_t(X)$ is the symmetric semi-martingale local time of $X$ at $z$. The second equality in \eqref{EQ1XTXB} is due to the occupation times formula (see \cite[Chapter VI. (1.6)]{RY99}). As a result, $(\EE,\FF)$ is linked with a weak solution to \eqref{EQ1YTW} with $\mu(dz)=\frac{\rho'(z)}{2\rho(z)}dz$. It is also possible to take an uncountinuous density function $\rho$. In \cite{ORT15} (as well as \cite{Ra11}), the unique solution to \eqref{EQ1YTW} is associated with the Dirichlet form $(\EE,\FF)$ where $\rho$ is a step function of the form 
\begin{equation}\label{EQ1RXN}
	\rho(x)=\sum_{n\geq 1} \gamma_n 1_{(a_n, b_n)}(x), 
\end{equation}
$\{\gamma_n: n\geq 1\}$ is a set of positive numbers determined by $\{\alpha^-_k, \alpha^+_k: k\in \bZ\}$ and $\{(a_n, b_n): n\geq 1\}$ is a set of disjoint open intervals such that $\cup_{n\geq 1}(a_n,b_n)=\bR\setminus \{z_p: p\in \bZ\}$.
Note incidentally that $\alpha$-skew Brownian motion corresponds to 
\[
	\rho(x)=\frac{1-\alpha}{\alpha}1_{(-\infty, 0)}(x)+1_{(0,\infty)}(x)
\]
for $\alpha\in (0,1)$.

The aim of this paper is twofold. On one hand, we will characterize when the distorted Brownian motion $X$ is a semi-martingale and derive its representation. It is in the hope that the expression \eqref{EQ1XTXB} can be extended to distorted Brownian motions with density functions in a family as wide as possible and particularly, the special situations mentioned above can be covered by a more general framework. On the other hand, we wish to find a general set of measures, with $\mu$ in which the SDE \eqref{EQ1YTW} has weak solutions associated with certain Dirichlet forms. It is also of interest to prove or disprove the pathwise uniqueness for \eqref{EQ1YTW}.


To study the distorted Brownian motion, the basic tool is the well-known Fukushima's decomposition. Note that $u(x):=x$ belongs to the local Dirichlet space $\FF_\text{loc}$ and then we can write the Fukushima's decomposition of $X$ relative to $u$:
\begin{equation}\label{EQ1XTX}
	X_t-X_0=M^u_t+N^u_t,
\end{equation}
where $M^u$ is a martingale additive functional (MAF in abbreviation) and $N^u$ is a continuous additive functional (CAF in abbreviation) locally of zero energy (see \cite[\S5.2]{FOT11}). One can easily deduce that $M^u$ is equivalent to a standard Brownian motion (see Proposition~\ref{PRO2}). So the challenge is to characterize $N^u$ and to figure out the connections between $N^u$ and symmetric semi-martingale local times. Related considerations to formulate $N^u$ are presented in \cite{Fu79, Fu99} as well as the references therein and collected in the book \cite[\S5]{FOT11}. It turns out that $N^u$ is of bounded variation, if and only if a certain smooth signed measure $\nu$ exists with $\EE(u,g)=\langle \nu, g\rangle$ for any function $g$ in a certain family. As a result,
\begin{equation}\label{EQ1NUT}
	N^u_t=-\int_\bR \ell^z_t \nu(dz),
\end{equation}
where $(\ell^z_t)_{t\geq 0}$ is the positive continuous additive functional (PCAF in abbreviation) of $\delta_z$ in their Revuz correspondence relative to $(\EE,\FF)$, i.e.  for any $f\in C_c(\bR)$, 
\[
	\delta_z(f)=\lim_{t\downarrow 0} \frac{1}{t}\int_\bR \fm(dx) \cdot \mathbf{E}_x \int_0^t f(X_s)d\ell^z_s. 
\]
There are sufficient conditions like $\rho$ is locally of bounded variation (see \cite{Fu99}) that can lead to the existence of $\nu$. But we are not satisfied with them. In \S\ref{SEC4}, one of the main results Theorem~\ref{THM3} will state an equivalent condition based on the irreducible decomposition of $X$. More precisely, $X$ is not necessarily irreducible, as shown in \S\ref{SEC21}. This means a proper subset $A$ of $\bR$ may exist such that $X$ will never leave $A$ if starts from a point in $A$. This set $A$ is called an invariant set of $X$ and the restriction $X|_A$ of $X$ to $A$ is also a Markov process (see \cite[\S2.1]{CF12}). By virtue of a representation theorem obtained in \cite{LY172}, $X$ can be characterized by a set of so-called effective intervals $\{(\tI_k, \fs_k): k\geq 1\}$, where $\{\tI_k: k\geq 1\}$ is a set of disjoint intervals and $\fs_k$ is an ``adapted" scale function on $\tI_k$. Each $\tI_k$ is an invariant set of $X$ and its restriction $X|_{\tI_k}$ to $\tI_k$ is an irreducible diffusion determined by $\fs_k$. After figuring out the expression of these effective intervals in \S\ref{SEC3}, we will conclude in Theorem~\ref{THM3} that $N^u$ is of bounded variation, if and only if for any $k\geq 1$, the restriction $\rho|_{\tI_k}$ of $\rho$ to $\tI_k$ is locally of bounded variation. Moreover, $-\nu|_{\tI_k}$ where $\nu$ is in \eqref{EQ1NUT} is the induced measure (on $\tI_k$) of $\rho|_{\tI_k}$. To link $N^u$ with symmetric semi-martingale local times, it will turn out in Lemma~\ref{LM6} that 
\begin{equation}\label{EQ1LZT}
	L^z_t(X)=\frac{\rho(z)+\rho(z-)}{2}\cdot \ell^z_t,
\end{equation}
where $\rho(z-)$ is the left limit of $\rho$ at $z$ (see Definition~\ref{DEF1}). This formula extends a result in \cite{ORT15}. As a consequence, one can find that $X$ is a weak solution to \eqref{EQ1YTW} with 
\begin{equation}\label{EQ1MZN}
	\mu(dz)=-\frac{\nu(dz)}{\rho(z)+\rho(z-)}.
\end{equation}
When $\rho$ satisfies the condition stated above, we shall call $X$ a \emph{general skew Brownian motion} (with the density function $\rho$) in Definition~\ref{DEF3}. 

As an outgrowth of the study concerning distorted Brownian motions, we will reconsider the SDE \eqref{EQ1YTW} and the main purpose is to find suitable conditions on $\mu$ such that there exist general skew Brownian motions being weak solutions to \eqref{EQ1YTW}. Since it suffices to construct a density function $\rho$ satisfying \eqref{EQ1MZN} for the given $\mu$, first of all, we must impose $|\mu(\{z\})|\leq 1$. Denote
\[
	G:=\{z\in \bR: \exists \varepsilon>0, |\mu|((z-\varepsilon,z+\varepsilon))<\infty\}. 
\]	
Clearly, $G$ is open and may be written as a union of disjoint open intervals: 
\[
	G=\cup_{n\geq 1}I_n=\cup_{n\geq 1}(a_n,b_n).
\]
Then a complete characterization of the existence of related skew Brownian motions will be phrased in Theorem~\ref{THM58}. It is shown that $\Xi^+:=\{z: \mu(\{z\})=1\}$ (resp. $\Xi^-:=\{z: \mu(\{z\})=-1\}$) has to be a subset of $\{a_n: n\geq 1\}$ (resp. $\{b_n: n\geq 1\}$) and one must further impose $|\mu|(G^c\setminus \Xi)=0$ where $\Xi:=\Xi^+\cup \Xi^-$. Particularly, $|\mu|$ is Radon on each $I_n$ with $|\mu(\{z\})|<1$ for any $z\in I_n$. Hence restricting to $I_n$, the density function $\rho$ is determined by $\mu$ uniquely up to a multiplicative constant in a classical manner \eqref{EQ5VIZ} (see also \cite[Lemma~2.1]{LG84}). The next crucial step is to obtain the effective intervals by gluing all so-called scale-connected intervals in $\{I_n:n\geq 1\}$. This will be accomplished in \S\ref{SEC6}. Since every condition in this characterization is both sufficient and necessary, the theory of Dirichlet forms has done its best to attain weak solutions to \eqref{EQ1YTW}. 

We declare general skew Brownian motions related to \eqref{EQ1YTW} to be unique if all of them are equivalent in distribution. When this uniqueness fails, there are obviously different weak solutions to \eqref{EQ1YTW} and occasionally, infinite ones can be found. An interesting example will be raised in Corollary~\ref{COR510}, where $G^c$ is a generalized Cantor set and $\mu=\sum_{n\geq 1}\delta_{a_n}-\sum_{n\geq 1}\delta_{b_n}$. It is also worth noting that this kind of uniqueness holds, if and only if every effective interval is ended by $a_n$ and $b_n$ for some $n$ (see Corollary~\ref{COR7}). More precisely, the set of effective intervals must be $\{\langle a_n, b_n\rangle: n\geq 1\}$, where $\langle a_n,b_n\rangle$ may be open, semi-open/semi-closed or closed. Moreover, $|\mu|$ is Radon on $\langle a_n,b_n\rangle$ and the restriction $X|_{\langle a_n,b_n\rangle}$ of $X$ to $\langle a_n,b_n\rangle$ is uniquely determined by $\mu$.  As explained in Remark~\ref{RM74}, there is a sense in which this uniqueness leads to the pathwise uniqueness of \eqref{EQ1YTW} by attaching a suitable lifetime to the weak solutions.

Three special cases of Theorem~\ref{THM58} will be explored further. The first one is $G^c=\emptyset$. In other words, $|\mu|$ is Radon on $\bR$ and $|\mu(\{z\})|<1$ for any $z\in \bR$. On one hand, we will indicate that there exists a unique general skew Brownian motion related to \eqref{EQ1YTW}. On the other hand, \eqref{EQ1YTW} is well posed as proved in Theorem~\ref{THM71} and hence its unique weak solution coincides with the general skew Brownian motion obtained above. It is worth noting that the situations appeared in \cite{LG84}, \cite{ORT15} and \cite{Ra11} are covered by this case. The second case assumes that $G^c=\Xi$ is a discrete set of countable points. A point $z \in \Xi^+$ (resp. $z\in \Xi^-$) is usually called a right (resp. left) \emph{barrier} as it is indicated in \cite{BE14} that when $|\mu|$ is Radon on $\bR$, the solution to \eqref{EQ1YTW} for $x\geq z$ (resp. $x\leq z$) cannot pass through $z$. However in our case, $|\mu|$ is not necessarily Radon on $\bR$ and the ``barriers" in $\Xi$ may play different roles. In practice, we will classify every point in $\Xi$ as a \emph{real, pseudo} or \emph{nonsensical barrier} in \S\ref{SEC53}. These names  come from the following facts stated in Theorem~\ref{THM7}: The presence of nonsensical barriers breaks the existence of general skew Brownian motions related to \eqref{EQ1YTW}; pseudo barriers are not ``real" because every related general skew Brownian motion can pass through them from both sides; and only real barriers are effective like the case $|\mu|$ is Radon on $\bR$. Finally, the third case involves a Cantor-type structure, i.e. $G^c$ is assumed to be a generalized Cantor set $K$. Without loss of generality, we further assume $\Xi^+=\{a_n: a_n>-\infty, n\geq 1\}$ and $\Xi^-=\{b_n: b_n<\infty, n\geq 1\}$. It is the case that every point in $G^c$ is an accumulation point of $\Xi$. The characterization of the existence of related general skew Brownian motion will be simplified in Theorem~\ref{THM8}. Particularly, when $K$ is produced by a sequence $\{\alpha_j: j\geq 1\}$ of numbers in $(0,1)$ with $\alpha_j\equiv \alpha\in (0,1)$ ($\alpha=1/3$ corresponds to the standard Cantor set), $\alpha< 1/4$ leads to the irreducibility of all related general skew Brownian motions and for $\alpha\geq 1/4$, they are unique. 

The rest of this paper is organized as follows. In the sections from \S\ref{SEC3} to \S\ref{SEC5}, we will study the distorted Brownian motion $X$ associated with the Dirichlet form $(\EE,\FF)$ given by the closure of \eqref{EQ2EFG}. The basic properties and the expression of the effective intervals of $X$ will be presented in \S\ref{SEC3}. The section \S\ref{SEC4} is mainly devoted to prove an equivalent condition of that $N^u$ in the Fukushima's decomposition \eqref{EQ1XTX} is of bounded variation. Particularly, the semi-martingale representation of $X$ will be obtained under the same condition. In \S\ref{SEC34}, the relation between $\ell^z$ and $L^z(X)$ will be figured out. As a result, we can conclude in \S\ref{SEC5} that the so-called general skew Brownian motion is a weak solution to \eqref{EQ1YTW} with $\mu$ in \eqref{EQ1MZN}. Note that we will review a representation theorem for regular and strongly local Dirichlet forms on $L^2(\bR,\fm)$ in the appendix \ref{SEC2} and their quasi notions like $\EE$-nest, $\EE$-polar set, $\EE$-quasi-continuous function and smooth measure will be characterized in \ref{SEC22}. These characterizations play an important role in proving the results mentioned above. The remainder sections are devoted to the exploration of the SDE \eqref{EQ1YTW}. After preparing a useful lemma in Lemma~\ref{LM58}, the main result characterizing the existence of related general skew Brownian motions will be stated in Theorem~\ref{THM58}. Moreover, Corollaries~\ref{COR67} and \ref{COR614} describe their irreducibility and equivalence in distribution. The three special cases mentioned above will be treated in \S\ref{SEC7}.  We should point out that the appearing processes in this part are assumed to be conservative, and the conditions \eqref{EQ5LXV2}, \eqref{EQ5XVX}, \eqref{EQ5LXV} and \eqref{EQ5XXX} are used only for guaranteeing this assumption in various situations.

\subsection*{Notations}
Let us put some often used notations here for handy reference, though we may restate their definitions when they appear.

The notation `$:=$' is read as `to be defined as'. For $\fa<\fb$, $\tI:=\langle \fa, \fb\rangle$ is an interval where $\fa$ or $\fb$ may or may not be contained in $\langle \fa, \fb\rangle$. The classes $C_c(\tI), C^1_c(\tI)$ and $C^\infty_c(\tI)$ denote the spaces of all continuous functions on $\tI$ with compact support, all continuously differentiable functions with compact support and all infinitely differentiable functions with compact support, respectively.
The restrictions of a measure $\mu$ and a function $f$ to an interval $\tI$ are denoted by $\mu|_\tI$ and $f|_\tI$ respectively. The notation $\langle \mu, f\rangle$ stands for the integration of $f$ with respect to $\mu$. 
 For a scale function $\fs$ (i.e. a continuous and strictly increasing function) on  $\tI$, $d\fs$ represents its associated measure on $\tI$. Given a scale function $\fs$ on $\tI$ and another function $f$ on $\tI$, $f\ll \fs$ means $f=g\circ \fs$ for an absolutely continuous function $g$ and $$\frac{df}{d\fs}:=g'\circ \fs,$$
where $g'$ is the derivative of $g$. The notation $\delta_z$ stands for the Dirac measure at $z\in \bR$. Given a difference of two positive measures $\mu=\mu^+-\mu^-$ with $\mu^+\perp \mu^-$, we say $\mu$ is Radon signed on $\tI$ if for any compact subinterval $K$ of $\tI$, $\mu|_K:=\mu^+|_K-\mu^-|_K$ is a finite signed measure on $K$. Set $|\mu|:=\mu^++\mu^-$.

Fix a Markov process $X=\left\{(X_t)_{t\geq 0}, (\mathbf{P}_x)_{x\in E}, \zeta\right\}$ associated with a Dirichlet form $(\EE,\FF)$ on $L^2(E,\fm)$, where $\mathbf{P}_x$ is the probability measure such that $\mathbf{P}_x(X_0=x)=1$ and $\zeta$ is the lifetime of $X$. The notation $\mathbf{E}_x$ is the expectation induced by $\mathbf{P}_x$. We shall write $(X_t, \mathbf{P}_x)$ for this Markov process if no confusions cause. However when $x$ is specific, $(X_t,\mathbf{P}_x)$ also stands for the stochastic process $(X_t)_{t\geq 0}$ under the probability measure $\mathbf{P}_x$. Let $A$ be an invariant set of $X$, then $X|_A$ stands for its restriction to $A$. 
All terminologies about Dirichlet forms and Markov processes are standard and we refer them to \cite{FOT11, CF12}.


\section{Distorted Brownian motions}\label{SEC3}

What we are concerned with is a special family of regular and strongly local Dirichlet forms whose associated diffusions are the so-called distorted Brownian motions. 
Take a positive function $\rho$ on $\bR$ satisfying the assumption:
\begin{itemize}
\item[(A)] \eqref{EQ2RLR} holds and $\rho$ vanishes a.e. on its singular set
\[
S(\rho)=\left\{x\in \mathbb{R}: \text{for any }\varepsilon>0, \int_{x-\varepsilon}^{x+\varepsilon}\frac{1}{\rho(y)}dy=\infty \right\}. 
\]
\end{itemize}
Set $\fm(dx):=\rho(x)dx$ henceforth. Then the quadratic form
\begin{equation}\label{EQ2DEC}
\begin{aligned}
	\mathcal{D}(\EE)&=C_c^\infty(\mathbb{R}),\\
	\EE(f,g)&=\frac{1}{2}\int_\mathbb{R}f'(x)g'(x)\rho(x)dx,\quad f,g\in \mathcal{D}(\EE)
\end{aligned}
\end{equation}
is closable on $L^2(\bR,\fm)$ by \cite[Theorem~3.1.6]{FOT11}. Denote its closure by $(\EE,\FF)$, which is clearly a regular and strongly local Dirichlet form on $L^2(\bR,\fm)$, and the associated diffusion process of $(\EE,\FF)$ by $(X_t,\mathbf{P}_x)$. This process is the so-called distorted Brownian motion. 

The distorted Brownian motion can be represented by a set of so-called effective intervals $\{(\tI_k,\fs_k): k\geq 1\}$ as reviewed in Theorem~\ref{THM0} (see also \cite[\S3.4]{LY172}). Roughly speaking, $\tI_k=\langle \fa_k, \fb_k\rangle$ are mutually disjoint intervals and $\fs_k$ is a continuous and strictly increasing function on $\tI_k$; each $\tI_k$ is an invariant set of $X$ and the restriction $X|_{\tI_k}$ of $X$ to $\tI_k$ is an irreducible diffusion with the scale function $\fs_k$. Hereafter, we always take a fixed point $\fe_k\in (\fa_k,\fb_k)$ and impose $\fs_k(\fe_k)=0$. Particularly, we can summarize the following results.

\begin{lemma}{(see \cite[\S3.4]{LY172})}\label{LM32}
Let $(\EE,\FF)$ be the closure of \eqref{EQ2DEC} and $\{(\tI_k,\fs_k): k\geq 1\}$ the set of its effective intervals. Then
\begin{itemize}
\item[(1)] $S(\rho)= \left(\cup_{k\geq 1} \mathring{\tI}_k\right)^c$ where $\mathring{\tI}_k=(\fa_k,\fb_k)$ is the interior of $\tI_k$ and particularly $S(\rho)$ is nowhere dense;
\item[(2)] $\fs_k$ is absolutely continuous and its derivative $\fs'_k>0$ a.e. on $\tI_k$ for any $k\geq 1$; 
\item[(3)] $\rho=1/\fs'_k$ a.e. on $\tI_k$ for any $k\geq 1$;
\item[(4)] $X$ is irreducible, if and only if $1/\rho\in L^1_\mathrm{loc}(\bR)$. 
\end{itemize}
\end{lemma}
\begin{remark}\label{RM2}
Let us explain the details to obtain $\{(\tI_k,\fs_k): k\geq 1\}$. Note that $S(\rho)$ is closed and hence $S(\rho)^c$ can be written as a union of disjoint open intervals: $S(\rho)^c=\cup_{k\geq 1}(\fa_k,\fb_k)$. Take a fixed point $\fe_k\in (\fa_k,\fb_k)$ and set $\fs_k(x):=\int_{\fe_k}^x\frac{1}{\rho(z)}dz$. Further let $\tI_k:=\langle \fa_k,\fb_k\rangle$, where the finite endpoint $\fa_k\in \tI_k$ (resp. $\fb_k\in \tI_k$) if and only if $\fs_k(\fa_k):=\lim_{x\downarrow \fa_k}\fs_k(x)>-\infty$ (resp. $\fs_k(\fb_k):=\lim_{x\uparrow \fb_k}\fs_k(x)<\infty$). Finally, we obtain the set of effective intervals $\{(\tI_k,\fs_k): k\geq 1\}$. 
\end{remark}

It is worth noting that $\left(\cup_{k\geq 1}\tI_k\right)^c$ is of zero $\fm$-measure since $\rho=0$ a.e. on $S(\rho)$. As a consequence, $\left(\cup_{k\geq 1}\tI_k\right)^c$ is $\fm$-polar relative to $(\EE,\FF)$ by Corollary~\ref{COR1}. However every single point in $\cup_{k\geq 1}\tI_k$ is of positive capacity. Moreover, when $\tI_k$ is bounded (i.e. $|\fa_k|+|\fb_k|<\infty$), the restriction of $X$ to $\tI_k$ is recurrent. Hence the explosion of $X$ is possibly happened only at the infinite endpoints of some effective interval. Particularly, the conservativeness of $X$ is characterized in Proposition~\ref{THM4}. 

\section{Semi-martingale representation}\label{SEC4}

The main purpose of this section is to study when the distorted Brownian motion associated with $(\EE,\FF)$ becomes a semi-martingale. To this end, consider the coordinate function $u(x)=x$ for any $x\in \bR$. Note that $u\in \FF_\mathrm{loc}$. Then the Fukushima's decomposition of $X$ relative to $u$ is written as follows: for q.e. $x\in \mathbb{R}$ and $\mathbf{P}_x\text{-a.s.}$, 
\begin{equation}\label{EQ4UTX}
	X_t-X_0=M^u_t+N^u_t,\quad 0\leq t<\zeta,
\end{equation}
where $M^u=(M^u_t)_{t\geq 0}$ is an MAF locally of finite energy, $N^u=(N^u_t)_{t\geq 0}$ is a CAF locally of zero energy, and $\zeta$ is the lifetime of $X$ (see \cite[\S5.5]{FOT11}). 

\subsection{Martingale additive functional}

Denote the energy measure of $M^u$, i.e. the Revuz measure of the predictable quadratic variation $\langle M^u\rangle$ of $M^u$ by $\mu_{\langle M^u\rangle}$ (see \cite{FOT11}).  

\begin{lemma}
It holds that
\begin{equation}\label{EQ4MMU}
	\mu_{\langle M^u\rangle}= 1_{\{\cup_{k\geq 1}\tI_k\}}\cdot \fm.
\end{equation}
\end{lemma}
\begin{proof}
For every $n\geq 1$, take $u_n\in C_c^\infty(\bR)$ such that $u_n=u$ on $(-n,n)$. Denote the MAF in the Fukushima's decomposition of $X$ relative to $u_n$ by $M^{u_n}$. 
For any $f\in C_c^\infty(\bR)$ with $\text{supp}[f]\subset (-n,n)$, it follows from \cite[Theorem~5.2.3]{FOT11} and Theorem~\ref{THM0} that 
\[
\begin{aligned}
	\int_\bR fd\mu_{\langle M^{u_n}\rangle}&= 2\EE(u_n,u_nf)-\EE(u_n^2,f) \\&
=\sum_{k\geq 1}\int_{\tI_k}f\cdot \left(\frac{du_n}{d\fs_k}\right)^2d\fs_k \\
	&=\int_{\cup_{k\geq 1}\tI_k} f(x) \fm(dx). 
\end{aligned}\]
Hence $\mu_{\langle M^{u_n}\rangle}=1_{\cup_{k\geq 1}\tI_k}\cdot \fm$ on $(-n,n)$. Note that $\mu_{\langle M^u\rangle}=\mu_{\langle M^{u_n}\rangle}$ on $(-n,n)$.  
This leads to \eqref{EQ4MMU}. 
\end{proof}

Then we can conclude the following description of $M^u$.

\begin{proposition}\label{PRO2}
For any $x\in \cup_{k\geq 1}\tI_k$, $M^u$ is equivalent to a one-dimensional standard Brownian motion up to $\zeta$ under the probability measure $\mathbf{P}_x$. In other words, there exists a one-dimensional standard Brownian motion $B_t$ under $\mathbf{P}_x$ such that
\begin{equation}\label{EQ4PXM}
M^u_t=B_t,\quad 0\leq t<\zeta. 
\end{equation}
For any $x\notin \cup_{k\geq 1}\tI_k$, it holds in the sense of $\mathbf{P}_x$-a.s., 
\[
	M^u_t=0,\quad \forall t\geq 0. 
\]
\end{proposition}
\begin{proof}
Note that \eqref{EQ4MMU} indicates that for $x\in\bR$ and $\mathbf{P}_x$-a.s.,
\[
\langle M^u\rangle_t=\int_0^t 1_{\{\cup_{k\geq 1}\tI_k\}}(X_s)ds,\quad  t\geq 0.
\]
In the case of $x\in \cup_{k\geq 1}\tI_k$, $\langle M^u\rangle_t=t\wedge \zeta$ and hence \eqref{EQ4PXM} holds by virtue of \cite[Chapter 5, (1.7)]{RY99}. However in the case of $x\notin \cup_{k\geq 1}\tI_k$, $\langle M^u\rangle_t=0$ for any $t\geq 0$. Therefore, $M^u_t\equiv 0$ for any $t\geq 0$. 
\end{proof}

\subsection{Zero energy part}\label{SEC42}


For each $k$, denote the inverse function of $\fs_k$ by $\ft_k$. More precisely, let 
\[
	\tJ_k:=\fs_k(\tI_k)=\{\fs_k(x): x\in \tI_k\}.
\] 
and then
\[
	\ft_k=\fs_k^{-1}: \tJ_k\rightarrow \tI_k. 
\]
Note that $\fs_k(\fe_k)=0$, $\tJ_k$ is also an interval and we denote it by $\tJ_k=\langle \fc_k, \fd_k\rangle$. Moreover, if $\fa_k\notin \tI_k$ and $\fa_k\neq -\infty$ (resp. $\fb_k\notin \tI_k$ and $\fb_k\neq \infty$), then $\fc_k=-\infty$ (resp. $\fd_k=\infty$). See \S\ref{SEC21}. 

\begin{lemma}
Fix $k\geq 1$. Then $\ft_k$ is absolutely continuous on $\tJ_k$ and for a.e. $y\in \tJ_k$, 
\begin{equation}\label{EQ4TKY}
	\ft'_k(y)=\rho(\ft_k(y)). 
\end{equation}	
Furthermore, $\ft'_k\in L^2_\mathrm{loc}(\tJ_k)$. 
\end{lemma}
\begin{proof}
For any $y_1, y_2\in \tJ_k$, 
\[
	\int_{y_1}^{y_2}\rho(\ft_k(y))dy=\int_{\ft_k(y_1)}^{\ft_k(y_2)}\rho(x)d\fs_k(x)=\int_{\ft_k(y_1)}^{\ft_k(y_2)}\rho(x)\cdot \frac{1}{\rho(x)}dx=\ft_k(y_2)-\ft_k(y_1).
\]
Since $|\ft_k(y_2)-\ft_k(y_1)|<\infty$, it follows that $\rho\circ \ft_k\in L^1_\mathrm{loc}(\tJ_k)$. This implies $\ft_k$ is absolutely continuous and \eqref{EQ4TKY} holds. Moreover,
\[
	\int_{y_1}^{y_2}\ft'_k(y)^2dy=\int_{y_1}^{y_2}\rho(\ft_k(y))^2dy=\int_{\ft_k(y_1)}^{\ft_k(y_2)}\rho^2\cdot \frac{1}{\rho}dx<\infty,
\]
since $\rho\in L^1_\mathrm{loc}(\mathbb{R})$. 
\end{proof}

A function $F$ is called locally of bounded variation on an interval $\tI=\langle \fa,\fb\rangle$ if of bounded variation on every compact subinterval of $\tI$. Throughout this paper, we always take its canonical version in the following sense if without other statements.
\begin{definition}\label{DEF1}
Let $F$ be a function locally of bounded variation on $\tI=\langle \fa,\fb\rangle$. The canonical version $\tilde{F}$ of $F$ is defined as follows: 
\[
	\tilde{F}(x):= \lim_{y\downarrow x} F(y),\quad x\in \tI\setminus \{\fb\}, 
\]
and if $\fb\in \tI$, $\tilde{F}(\fb):=0$. The left limits of $\tilde{F}$ are
\[
	\tilde{F}(x-):=\lim_{y\uparrow x}F(y),\quad x\in \tI\setminus \{\fa\},
\]
and if $\fa\in \tI$, $\tilde{F}(\fa-):=0$. Set further $\tilde{F}^*(x):=\tilde{F}(x)-\tilde{F}(x-)$. 
\end{definition}
For the sake of brevity, the canonical version of $F$ is still denoted by $F$. Restricting to every compact subinterval of $\tI$, $F$ induces a finite signed measure. By applying the Jordan decomposition to these signed measures, one can obtain two positive Radon measures $\nu_F^+$ and $\nu_F^-$ on $\tI$ with $\nu_F^+\perp \nu_F^-$. In abuse of notion, we call
\[
	\nu_F:=\nu^+_F-\nu^-_F
\]
the \emph{Radon signed measure} induced by $F$ (though $\nu_F$ may be not a signed measure) in the sense that for any compact interval $K\subset \tI$, $\nu_F|_{K}$ is the finite signed measure on $K$ induced by $F$. Write $|\nu_F|:=\nu^+_F+\nu^-_F$ which is a positive Radon measure on $\tI$. Given a $|\nu_F|$-integral function $f$, set
\[
	\int_\tI fd\nu_F:=\int_\tI fd\nu^+_F-\int_\tI fd\nu^-_F. 
\]


\begin{remark}\label{LM4}
Note that for any $x\in \tI$, $\nu_{F}(\{x\})=F^*(x)$. Particularly, the non-zero set $D_{F}:=\{x\in \tI: F^*(x)\neq 0\}$ of $F^*$ (i.e. the set of discontinuous points of $F$) is countable, and for any compact set $K\subset \tI$, 
\[
	\sum_{x\in K} |F^*(x)|<\infty. 
\]
Except for the discrete part $\sum_{x\in \tI}F^*(x)\delta_x$, $\nu_{F}$ may also contain an absolutely continuous part and a singular continuous part (with respect to the Lebesgue measure), see \cite[\S3.5]{F99}. 
\end{remark}


Now we have a position to characterize when $N^u$ is of bounded variation in the sense that for any $t<\zeta$, $N^u$ is of bounded variation on $[0,t]$ and derive its expression. 

\begin{theorem}\label{THM3}
Assume (A). The following conditions are equivalent:
\begin{itemize}
\item[(1)] The zero energy part $N^u$ in \eqref{EQ4UTX} is of bounded variation.
\item[(2)] For any $k\geq 1$, an a.e. version of $\ft'_k$ is a right continuous function locally of bounded variation on $\tJ_k$.
\item[(3)] For any $k\geq 1$, an a.e. version of $\rho|_{\tI_k}$ is a right continuous function locally of bounded variation on $\tI_k$. 
\end{itemize}
In this case,  denote the canonical version of $\rho|_{\tI_k}$ by $\rho_k$ and the induced Radon signed measure of $\rho_k$ by $\nu_{\rho_k}$. Then for any $x\in \bR$, 
\begin{equation}\label{EQ4NUT2}
N^u_t=\frac{1}{2}\sum_{k\geq 1}\int_{\tI_k}\ell^z_t\nu_{\rho_k}(dz),\quad 0\leq t<\zeta,\quad \mathbf{P}_x\text{-a.s.},
\end{equation}
where $\ell^z=(\ell^z_t)_{t\geq 0}$ is the local time of $X$ at $z$, i.e. the PCAF of the smooth measure $\delta_z$ relative to $X$.
\end{theorem}
\begin{proof}
Note that by \cite[Theorem~5.5.4]{FOT11}, $N^u$ is of bounded variation, if and only if there exists a smooth signed measure $\nu$ and an $\EE$-nest $\{K_m:m\geq 1\}$ of compact sets associated with $\nu$ such that
\begin{equation}\label{EQ4EUF}
	\EE(u,f)=\langle\nu, f\rangle,\quad \forall f\in \cup_{m\geq 1}\FF_{b, K_m}. 
\end{equation}
Write $\nu=\nu^+-\nu^-$ for the Jordan decomposition of $\nu$, where $\nu^+$ and $\nu^-$ are both positive smooth measures associated with $\{K_m:m\geq 1\}$. 

\emph{(1) $\Rightarrow$ (2)}. 
Suppose $N^u$ is of bounded variation and $\nu, \{K_m: m\geq 1\}$ are given above. Fix $k\geq 1$ and take $\varphi\in C_c^\infty(\bR)$ with $\text{supp}[\varphi]\subset \mathring{\tJ}_k=(\fc_k, \fd_k)$.  Set $f:=\varphi\circ \fs_k$. Then $\text{supp}[f]\subset (\fa_k,\fb_k)$. It follows from Theorem~\ref{THM2}~(4) that $f\in \FF_{b,K_m}$ for some $m$. Thus
\[
\EE(u,f)=\langle \nu, f\rangle.
\]
Let $\tilde{\nu}^{\ft,\pm}_k$ be the image measure of $\nu^\pm|_{\tI_k}$ under the map $\fs_k$. 
Set $\tilde{\nu}^{\ft}_k:=\tilde{\nu}^{\ft,+}_k-\tilde{\nu}^{\ft,-}_k$ which is a finite signed measure when restricting to each compact subset of $\tJ_k$. Then one can find a function locally of bounded variation on $\tJ_k$ whose induced Radon signed measure is $\tilde{\nu}^{\ft}_k$. Denote the canonical version of this function by $F$. Note that $\nu_F$ coincides with $\tilde{\nu}^{\ft}_k$ on $\mathring{\tJ}_k$. 
It follows from
\[
	\EE(u,f)=\frac{1}{2}\int_{\tI_k}\frac{du}{d\fs_k}\varphi'\circ \fs_kd\fs_k=\frac{1}{2}\int_{\tJ_k}\ft'_k(y)\varphi'(y)dy
\]
and 
\[
	\langle \nu,f\rangle=\langle \tilde{\nu}^\ft_k, \varphi\rangle=\langle \nu_F,\varphi\rangle=-\int_{\fc_k}^{\fd_k}F(y)\varphi'(y)dy
\]
that 
\[
	\int_{\tJ_k}\varphi'(y)\ft'_k(y)dy=-2\int_{\fc_k}^{\fd_k}F(y)\varphi'(y)dy. 
\]
This leads to 
\[
	\int_{\fc_k}^{\fd_k}\varphi'(y)\left(\ft'_k(y)-2F(y)\right)dy=0, \quad \forall \varphi \in C_c^\infty((\fc_k, \fd_k)). 
\]
Hence we can conclude that $\ft'_k(y)-2F(y)\equiv C$ for some constant $C$ and a.e. $y\in (\fc_k, \fd_k)$. Particularly, an a.e. version of $\ft'_k$ is right continuous and  locally of bounded variation on $\tJ_k$. 

\emph{(2) $\Rightarrow$ (3).} It follows from \eqref{EQ4TKY} that 
\[
	\rho(x)=\ft'_k(\fs_k(x)),\quad \text{a.e. } x\in \tI_k. 
\]
Since absolutely continuous and strictly increasing by Lemma~\ref{LM32}, $\fs_k$ is a homeomorphism from $\tI_k$ to $\tJ_k$. Hence the second condition implies the third one. 

\emph{(3) $\Rightarrow$ (1).} Take such an a.e. version of $\rho|_{\tI_k}$ and denote it by $\rho_k$. Its induced Radon signed measure on $\tI_k$ is further denoted by $\nu_{\rho_k}=\nu_{\rho_k}^+-\nu_{\rho_k}^-$. Set
\begin{equation}\label{EQ4NKN}
\nu^-=\frac{1}{2}\sum_{k\geq 1}\nu^+_{\rho_k},\quad \nu^+=\frac{1}{2}\sum_{k\geq 1}\nu^-_{\rho_k},\quad \nu:=\nu^+-\nu^-.
\end{equation}
We can easily check that $\nu^\pm$ charges no Borel subsets of $(\cup_{k\geq 1}\tI_k)^c$ and $\nu^\pm|_{\tI_k}$ is a positive Radon measure on $\tI_k$. Thus $\nu$ is a smooth signed measure relative to $X$ by Corollary~\ref{COR1}. For each $k$, take an increasing sequence of compact intervals $\{F^k_m: m\geq 1\}$ such that $\cup_{m\geq 1}F^k_m=\tI_k$ and the closed endpoints of $\tI_k$ are contained in $F^k_m$. Then $\{F^k_m: m\geq 1\}$ is an $\EE^{(\fs_k)}$-nest. Write
\[
K_m:=\cup_{k=1}^mF^k_m.
\]
We know from Corollary~\ref{COR1} that $\{K_m: m\geq 1\}$ is an $\EE$-nest of compact sets associated with $\nu$. 
Fix $m$ and $f\in \FF_{b,K_m}$. Let $f_k:=f|_{\tI_k}$ for $1\leq k\leq m$.  It follows that
\[
	2\EE(u,f)=\sum_{k=1}^m \int_{\tI_k} f'_k(x)\rho_k(x)dx.
\]
Note that $C_c^\infty(\tI_k)$ is a special standard core of $(\EE^{(\fs_k)}, \FF^{(\fs_k)})$ (see \cite[Theorem~3.7]{LY172}). Thus for any $1\leq k\leq m$, we can take a compact interval $W$ with $F^k_m\subset W\subset \tI_k$ and a sequence $\{g^k_p: p\geq 1\}\subset C_c^\infty(\tI_k)$ with $\text{supp}[g^k_p]\subset W$ such that $g^k_p\rightarrow f_k$ in the $\EE^{(\fs_k)}_1$-norm. Particularly, $g^k_p$ converges to $f_k$ uniformly on $F^k_m$. 
On the other hand, it follows from \cite[Theorem~3.36 and Exercise 34(b)]{F99} that for any $g\in C_c^\infty(\tI_k)$, 
\[
	\int_{\tI_k} g'(x)\rho_k(x)dx=-\int_{\tI_k} gd\nu_{\rho_k}. 
\]
Applying this formula to $g^k_p$ and letting $p\rightarrow\infty$, we obtain
\[
\int_{\tI_k}f'_k(x)\rho_k(x)dx=-\int_{\tI_k} f_kd\nu_{\rho_k}
\]
Therefore, $\EE(u,f)=\langle \nu, f\rangle$ and $N^u$ is of bounded variation.

Finally, \eqref{EQ4NUT2} is implied by \cite[Theorem~5.5.4]{FOT11} and \eqref{EQ4NKN}. That completes the proof.
\end{proof}

\begin{remark}\label{RM3}
Let us give some remarks for this theorem.
\begin{itemize}
\item[(1)] Consider $x\in \tI_k$ for some $k\geq 1$. We have $\mathbf{P}_x(\ell^z_t\equiv 0,\forall t)=1$ for any $z\notin \tI_k$ (see \cite[Lemma~5.1.11]{FOT11}). Thus in fact it holds in \eqref{EQ4NKN} that $\mathbf{P}_x$-a.s.,
 \[
N^u_t=\frac{1}{2}\int_{\tI_k}\ell^z_t\nu_{\rho_k}(dz),\quad  0\leq t<\zeta.
 \]
 Otherwise if $x\notin \cup_{k\geq 1}\tI_k$, then $\mathbf{P}_x(N^u_t\equiv 0,\forall t)=1$. 
\item[(2)] The function $\rho_k$ could replaced by a different version. For example, suppose $\tI_k=[\fa_k, \fb_k]$, take two constants $c_1,c_2\neq 0$ and let $\tilde{\rho}_k(\fa_k-)=c_1$ and
\[
	\tilde{\rho}_k(x):=\left\{\begin{aligned}
	&\rho_k(x),\quad \fa_k\leq x<\fb_k,\\
	&c_2,\quad x= \fb_k. 
	\end{aligned} \right.
\]
Then $\tilde{\rho}_k$ is still a right continuous function (locally) of bounded variation. Denote its induced Radon signed measure by $\nu_{{\tilde{\rho}}_k}$. We find
\[
\begin{aligned}
\nu_{\rho_k}=\nu_{{\tilde{\rho}}_k}+c_1\delta_{\fa_k}-c_2\delta_{\fb_k}. 
\end{aligned}
\]
For any $x\in \tI_k$, it holds that $\mathbf{P}_x$-a.s.,
 \[
N^u_t=\frac{1}{2}\int_{\tI_k}\ell^z_t\nu_k^{\tilde\rho}(dz)+\frac{c_1}{2}\ell^{\fa_k}_t-\frac{c_2}{2}\ell^{\fb_k}_t,\quad 0\leq t< \zeta.
\]
For the sake of brevity, the canonical version $\rho_k$ was chosen. 
\item[(3)] The equivalence between the first and second conditions has been studied in \cite{FL16} for a simple case that $X$ is an irreducible diffusion on an open interval. 
\item[(4)] The case that $\rho$ is locally of bounded variation on $\bR$ is treated in \cite[Theorem~7.1]{Fu99}. That result by Fukushima is valid not only for one-dimensional distorted Brownian motions but also for multi-dimensional ones. 
\end{itemize}
\end{remark}


When the equivalent conditions in Theorem~\ref{THM3} hold, we may assume without loss of generality that $\rho(x)=\rho_k(x)$ on $\tI_k$ for any $k\geq 1$ and $\rho(x)=0$ for $x\notin \cup_{k\geq 1}\tI_k$. Further set for any $k\geq 1$ and $x\in \tI_k$,
\[
	\rho(x-):=\rho_k(x-)
\]
and otherwise if $x\notin \cup_{k\geq 1}\tI_k$, set $\rho(x-):=0$. Let $\rho^*(x):=\rho(x)-\rho(x-)$. Finally, define $\nu_\rho:=\sum_{k\geq 1}\nu_{\rho_k}$ and $|\nu_\rho|:=\sum_{k\geq 1}|\nu_{\rho_k}|$ for later use. 



\subsection{Semi-martingale representation}\label{SEC44}

Eventually we can present the semi-martingale decomposition of $X$. Note that $(X_t, \mathbf{P}_x)$ is called a semi-martingale up to $\zeta$ if provided that there exists a sequence of stopping times $\sigma_n$ increasing to $\zeta$ such that $X^{\sigma_n}_t:=X_{t\wedge \sigma_n}$ is a semi-martingale. 

\begin{corollary}\label{THM5}
Assume (A) and that an a.e. version of $\rho|_{\tI_k}$ is locally of bounded variation on $\tI_k$ for any $k\geq 1$. Then $X$ is a semi-martingale up to $\zeta$ under the probability measure $\mathbf{P}_x$ for any $x\in \mathbb{R}$.
Meanwhile, the semi-martingale decomposition of $X$ is as follows: for any $x\in \cup_{k\geq 1}\tI_k$, there exists a standard Brownian motion $B=(B_t)_{t\geq 0}$ under the probability measure $\mathbf{P}_x$ such that
\begin{equation}\label{EQ4XTXB}
X_t-x=B_t+\frac{1}{2}\int_{\bR}\ell^z_t\nu_{\rho}(dz),\quad 0\leq t<\zeta,\quad \mathbf{P}_x\text{-a.s.},
\end{equation}
where $\ell^z$ is the local time of $X$ at $z$ and $\nu_{\rho}$ is given in the end of \S\ref{SEC42}; otherwise if $x\notin \cup_{k\geq 1}\tI_k$, $X_t\equiv x$ for any $t\geq 0$. 
\end{corollary}
\begin{proof}
For $x\notin \cup_k\tI_k$, $X_t\equiv x$ is clearly a semi-martingale. Now consider $x\in \cup_k\tI_k$. It follows from Proposition~\ref{PRO2} and Theorem~\ref{THM3} that \eqref{EQ4XTXB} holds. For $T>0$ and a.s. $\omega$ with $T<\zeta(\omega)$,
\[
	t\mapsto N^u_t(\omega)
\] 
is of bounded variation on $[0,T]$. 
Let $\sigma_n:=\inf\{t>0: X_t\notin (-n,n)\}$. Then $\mathbf{P}_x(\lim_{n\uparrow \infty}\sigma_n=\zeta)=1$ by \cite[Lemma~5.5.2]{FOT11} and $X^{\sigma_n}$ is a semi-martingale. That completes the proof. 
\end{proof}
\begin{remark}
\begin{itemize}
\item[(1)] When $X$ is conservative, i.e. $\mathbf{P}_x(\zeta=\infty)=1$ for all $x\in \bR$ (see Proposition~\ref{THM4} in the appendix), \eqref{EQ4XTXB} holds for all $t\geq 0$ and $X$ is a semi-martingale under $\mathbf{P}_x$ for all $x\in \bR$.
\item[(2)] When $\nu_\rho$ is discrete, \eqref{EQ4XTXB} can be written as
\[
	X_t-x=B_t+\frac{1}{2}\sum_{z\in D_\rho} \rho^*(z)\cdot \ell^z_t,\quad 0\leq t<\zeta,\quad \mathbf{P}_x\text{-a.s.},
\]
where $D_\rho=\{z:\nu_\rho(\{z\})\neq 0\}$.
\end{itemize}
\end{remark}

\subsection{Examples}

We give an interesting example where $\nu_\rho$ charges all rational numbers.

\begin{example}
Let $\bQ_+=\{q^+_n: n\geq 1\}$ (resp. $\bQ_-:=\{q^-_n: n\geq 1\}$) be the set of all positive (resp. negative) rational numbers. Take two sequences of positive constants $\{\varrho_n^+: n\geq 1\}$ and $\{\varrho_n^-: n\geq 1\}$ such that $\sum_{n\geq 1} (\varrho_n^++\varrho_n^-)<\infty$. Set
\[
	\begin{aligned}
		\rho(x):=\left\lbrace 
		\begin{aligned}
			& 1+\sum_{n: q^+_n\in (0,x]}\varrho^+_n, \quad x\geq 0,\\
			& 1+\sum_{n: q^-_n\in (x,0]}\varrho^-_n,\quad x<0.
		\end{aligned}
		\right. 
	\end{aligned}
\]
Then one can verify that the distorted Brownian motion $X$ is irreducible and conservative. Clearly, $\rho$ is of bounded variation and its induced measure $\nu_\rho$ is discrete. Thus $X$ is a semi-martingale and
\[
	X_t-x=B_t+\frac{1}{2}\sum_{n\geq 1} \left(\varrho_n^+\cdot \ell^{q^+_n}_t-\varrho_n^-\cdot \ell^{q^-_n}_t \right),\quad t\geq 0,\quad \mathbf{P}_x\text{-a.s.}
\]
for any $x\in \bR$. 
\end{example}

Another example below presents a distorted Brownian motion but not a semi-martingale. 

\begin{example}
Let $A:=\cup_{k\geq 1}(\frac{1}{2k+1}, \frac{1}{2k})$ and take
\[
	\rho(x)=2\cdot 1_A(x)+1_{A^c}(x). 
\]
Since $\rho$ is bounded below and above, one can easily verify that (A) holds and the Dirichlet form with the density function $\rho$ is irreducible and conservative. However, $\rho$ is clearly not locally of bounded variation on $\bR$. Hence $X$ is not a semi-martingale.
\end{example}

\section{Local times}\label{SEC34}

In this section, we always impose the conditions in Corollary~\ref{THM5}. Consequently, $X$ is a semi-martingale up to $\zeta$ under $\mathbf{P}_x$ for any $x\in \bR$.

Fix $k$ and $z\in \tI_k$. In the semi-martingale representation of $X$ stated above, the local time $\ell^z$ is the PCAF of $\delta_z$ relative to $X$. More precisely, it is uniquely determined by the so-called Revuz correspondence: for any $f\in C_c(\bR)$, 
\[
	\delta_z(f)=\lim_{t\downarrow 0} \frac{1}{t}\int_\bR \fm(dx) \cdot \mathbf{E}_x \int_0^t f(X_s)d\ell^z_s. 
\]
Particularly, $\ell^z$ depends on the symmetric measure $\fm$ but is independent of the starting point $x$ of $X$. We will also write $\ell^z(\fm)$ for $\ell^z$ when there is a risk of ambiguity. Clearly, $d\ell^z_t$ is a.s. carried by $\{t: X_t=z\}$ (see \cite[Lemma~5.1.11]{FOT11}). Furthermore, $\ell^z$ is not trivial in the following sense. 

\begin{lemma}\label{LM35}
Take $z\in \tI_k$ and set $R_z:=\inf\{t>0: \ell^z_t>0\}$. Then for any $x\in \tI_k$, it holds
\[
	\mathbf{P}_x(R_z<\infty)>0. 
\]
\end{lemma}  
\begin{proof}
Note that the quasi support of $\delta_z$ is $\{z\}$. Then it follows from \cite[Lemma~5.1.11]{FOT11} that $\mathbf{P}_x(R_z=\sigma_z)=1$ where $\sigma_z:=\inf\{t>0: X_t=z\}$. Therefore $\mathbf{P}_x(R_z<\infty)=\mathbf{P}_x(\sigma_z<\infty)>0$ since the restriction of $X$ to $\tI_k$ is an irreducible diffusion. 
\end{proof}

\begin{remark}
When $(\EE^{(\fs_k)},\FF^{(\fs_k)})$ is recurrent, it holds that $\mathbf{P}_x(\sigma_z<\infty)=1$ for $x,z\in \tI_k$ (see \cite[pp.124]{CF12}). Hence we also have $\mathbf{P}_x(R_z<\infty)=1$. Note that the recurrence of $(\EE^{(\fs_k)},\FF^{(\fs_k)})$ is characterized in \cite[\S2.2.3]{CF12}. Particularly, if $\tI_k$ is bounded in our case, then $(\EE^{(\fs_k)},\FF^{(\fs_k)})$ is recurrent.
\end{remark}


Another local time appearing in \eqref{EQ1YTW} is the so-called symmetric semi-martingale local time  (see also the celebrated Tanaka formula such as in \cite[Chapter VI. \S1]{RY99}). Fix $x\in \tI_k$ and consider the semi-martingale $(X_t, \mathbf{P}_x)$ up to $\zeta$. Denote the symmetric semi-martingale local time of $(X_t, \mathbf{P}_x)$ at $z$ by $(L^z_t(X,x))_{t\geq 0}$ ($(L^z_t(X))_{t\geq 0}$ in abbreviation if no confusions cause), i.e. for $z\in \bR$, 
\[
	L^z_t(X)=\lim_{\varepsilon\downarrow 0} \frac{1}{2\varepsilon}\int_0^t 1_{(z-\varepsilon, z+\varepsilon)}(X_s)d\langle X\rangle_s,\quad 0\leq t<\zeta,\quad \mathbf{P}_x\text{-a.s.}, 
\]
where $\langle X\rangle$ is the quadratic variation process of $(X_t, \mathbf{P}_x)$. Particularly, $L^z(X)\equiv 0$ for $z\notin \tI_k$ and when $\fa_k\in \tI_k$ or $\fb_k\in \tI_k$,
\[
	L^{\fa_k}_t(X)=\lim_{\varepsilon\downarrow 0} \frac{1}{2\varepsilon}\int_0^t 1_{[\fa_k, \fa_k+\varepsilon)}(X_s)d\langle X\rangle_s,\quad 0\leq t<\zeta,\quad \mathbf{P}_x\text{-a.s.}, 
\]
or
\[
	L^{\fb_k}_t(X)=\lim_{\varepsilon\downarrow 0} \frac{1}{2\varepsilon}\int_0^t 1_{(\fb_k-\varepsilon, \fb_k]}(X_s)d\langle X\rangle_s,\quad 0\leq t<\zeta, \quad  \mathbf{P}_x\text{-a.s.}.
\]
Note that $L^z(X)$ is independent of $\fm$ but depends on the starting point $x$. 

We need a lemma to link $L^z(X)$ with $\ell^z$. 

\begin{lemma}\label{LM6}
Let $X$ be the semi-martingale \eqref{EQ4XTXB}. Fix $k\geq 1$ and take $x,z\in \tI_k$. The local time $\ell^z$ associated with $\delta_z$ (under the symmetric measure $\fm$) and the symmetric semi-martingale local time $L^z(X)$ of $(X_t, \mathbf{P}_x)$ at $z$ are given as above. Then it holds
\begin{equation}\label{EQ3LZT}
	L^z_t(X)=\frac{\rho(z)+\rho(z-)}{2}\cdot \ell^z_t,\quad 0\leq t<\zeta, \quad \mathbf{P}_x\text{-a.s.}
\end{equation}
\end{lemma}
\begin{proof}
Note that $L^z_t(X)$ is the unique increasing process such that 
\begin{equation}\label{EQ3XTZ}
	|X_t-z|=|x-z|+\int_0^t \mathrm{sgn}(X_s-z)dX_s + L^z_t(X),\quad 0\leq t<\zeta,
\end{equation}
where $\mathrm{sgn}(a)$ is taken to be $1$ for $a>0$, $-1$ for $a<0$ and $0$ for $a=0$. Indeed, by applying Tanaka formula (see \cite[Chapter VI. (1.25)]{RY99}) to $X^{\sigma_n}$ where $\sigma_n$ is given in the proof of Corollary~\ref{THM5}, one can find \eqref{EQ3XTZ} holds for $t<\sigma_n$. We obtain \eqref{EQ3XTZ} by letting $n\uparrow \infty$. Let $\nu_{\rho_k}$ be in Corollary~\ref{THM5}. Then \eqref{EQ4XTXB} and Remark~\ref{RM3}~(1) tell us 
\[
	dX_t=dB_t+\frac{1}{2}\int_{y\in \tI_k}\nu_{\rho_k}(dy) \cdot d\ell^y_t,\quad 0\leq t<\zeta.
\]
Substituting it in \eqref{EQ3XTZ}, we obtain 
\begin{equation}\label{EQ3XTZX}
\begin{aligned}
	|X_t-z|=|x-z|+&\int_0^t \mathrm{sgn}(X_s-z)dB_s \\
	&+\frac{1}{2}\int_{y>z, y\in \tI_k}\ell^y_t \nu_{\rho_k}(dy)-\frac{1}{2}\int_{y<z,y\in \tI_k}\ell^y_t \nu_{\rho_k}(dy) + L^z_t(X).
\end{aligned}\end{equation}
Note that $t\mapsto \int_0^t \mathrm{sgn}(X_s-z)dB_s$ is also a standard Brownian motion. 

Let $f(y):=|y-z|$ for any $y\in \bR$. One may easily check that $f\in \FF_\mathrm{loc}$. Thus we can write the Fukushima's decomposition relative to $f$:
\begin{equation}\label{EQ3FXT}
	f(X_t)-f(x)=M^f_t+N^f_t,\quad  0\leq t<\zeta,\quad \mathbf{P}_x\text{-a.s.},
\end{equation}
where $M^f$ is an MAF and $N^f$ is a CAF locally of zero energy. Mimicking Proposition~\ref{PRO2}, we can deduce that $M^f$ is equivalent to a standard Brownian motion. On the other hand, mimicking the proof of Theorem~\ref{THM3}, one obtains
\begin{equation}\label{EQ3NFT}
	N^f_t=\frac{1}{2}\int_{y>z, y\in \tI_k}\ell^y_t \nu_{\rho_k}(dy)-\frac{1}{2}\int_{y<z,y\in \tI_k}\ell^y_t \nu_{\rho_k}(dy) +\frac{\rho_k(z)+\rho_k(z-)}{2}\cdot \ell^z_t. 
\end{equation}
Eventually from \eqref{EQ3XTZX} \eqref{EQ3FXT} and \eqref{EQ3NFT} we can conclude \eqref{EQ3LZT}. 
\end{proof}
\begin{remark}\label{RM5}
A special case of Lemma~\ref{LM6} appears in \cite[(2.12)]{ORT15} for the situation that $\rho$ is given by \eqref{EQ1RXN}. 
\end{remark}

The relation \eqref{EQ3LZT} indicates that if $\rho(z)=\rho(z-)=0$, then $L^z(X)\equiv 0$ (though $\ell^z$ is not trivial by Lemma~\ref{LM35}). This also leads to the a.s. continuity of $y\mapsto L^y_t(X)$ at $y=z$ (see \cite[Chapter VI. (1.7)]{RY99}). A concrete example is given as follows. 

\begin{example}\label{EXA1}
Take $\rho(x)=|x|^{\alpha}$ for any $x\in \bR$ with a constant $0<\alpha<1$. One may check that (A) holds and the closure $(\EE,\FF)$ of \eqref{EQ2DEC} is a regular and strongly local Dirichlet form on $L^2(\bR,\fm)$. Since $1/\rho\in L^1_\mathrm{loc}(\bR)$, it follows from Lemma~\ref{LM32}~(4) that $(\EE,\FF)$ is irreducible. Proposition~\ref{THM4} leads to its conservativeness. Therefore, the associated diffusion $X$ is a conservative and irreducible diffusion on $\bR$. Its scale function is equal to
\[
	\fs(x)=\int_0^x \frac{1}{\rho(y)}dy=\left\lbrace\begin{aligned}
	 &\frac{|x|^{1-\alpha}}{1-\alpha},\quad x\geq 0, \\
	 &-\frac{|x|^{1-\alpha}}{1-\alpha},\quad x< 0,
	\end{aligned} \right.
\]
and its speed measure is $\fm(dx)=\rho(x)dx=|x|^\alpha dx$. 

Since $\rho$ is clearly absolutely continuous and hence locally of bounded variation, we can conclude that $X$ is a semi-martingale under $\mathbf{P}_x$ for any $x\in \bR$. But $\rho(0)=\rho(0-)=0$. Hence \eqref{EQ3LZT} tells us 
\[
	\mathbf{P}_x(L^0_t(X)\equiv 0,\forall t)=1
\]
for any $x\in \bR$, while $\ell^0$ is not trivial by Lemma~\ref{LM35}.
\end{example}

\section{General skew Brownian motions}\label{SEC5}

For the sake of convenience, we introduce the following definition. Note that the conservativeness of $X$ is characterized in Proposition~\ref{THM4}. 

\begin{definition}\label{DEF3}
Under the same assumptions as Corollary~\ref{THM5}, let $X$ be the semi-martingale \eqref{EQ4XTXB} (up to $\zeta$). When $X$ is conservative, we call it a \emph{general skew Brownian motion}. The function $\rho$ is called the density function of $X$.
\end{definition}

With the formula \eqref{EQ3LZT} at hand, one can find that a general skew Brownian motion is always a weak solution to \eqref{EQ3YTT} with certain $\mu$. 

\begin{lemma}\label{COR10}
Under the same conditions as Corollary~\ref{THM5}, assume further that $(\EE,\FF)$ is conservative. Then $Z_\rho:=\{z\in \cup_{k\geq 1}\tI_k: \rho(z)=\rho(z-)=0\}$ is of zero $|\nu_\rho|$-measure. Particularly, for any $x\in \cup_{k\geq 1}\tI_k$, $(X_t, \mathbf{P}_x)$ is a weak solution to \eqref{EQ1YTW} with
\begin{equation}\label{EQ3MZN}
	\mu(dz)=\frac{\nu_\rho(dz)}{\rho(z)+\rho(z-)}. 
\end{equation}
\end{lemma} 
\begin{proof}
It suffices to show $|\nu_\rho|(Z_\rho)=0$. Denote the closure of $Z_\rho$ by $\bar{Z}_\rho$. Fix $k\geq 1$ and write $(\fa_k,\fb_k)\setminus \bar{Z}_\rho$ as a union of disjoint open intervals: 
\begin{equation}\label{EQ3AKB}
	(\fa_k,\fb_k)\setminus \bar{Z}_\rho=\cup_{p\geq 1}(c_p,d_p).
\end{equation}
For $c_p\in (\fa_k,\fb_k)$, we have $\rho(c_p-)=0$ since $c_p\in \bar{Z}_\rho$. Similarly we can also obtain $\rho(d_p)=0$ for $d_p\in (\fa_k,\fb_k)$. For any $z\in \left(\bar{Z}_\rho\cap (\fa_k,\fb_k)\right)\setminus \{c_p,d_p:p\geq 1\}$, one can take a subsequence of $\{d_p: p\geq 1\}$ that increases or decreases to $z$. Hence $\rho(z)=\rho(z-)=0$. Therefore $\nu_\rho|_{(\fa_k,\fb_k)}=\sum_{p\geq 1}\nu_\rho|_{[c_p, d_p]}$. Note that $\nu_\rho(\{z\})=0$ for all $z\in Z_\rho$. Eventually we can conclude $|\nu_\rho|(Z_\rho)=0$. 
\end{proof}

In Le Gall's paper \cite{LG84}, the well-posedness of \eqref{EQ3YTT} requires that $\mu$ is a finite signed measure and $|\mu(\{z\})|<1$ for any $z\in \bR$. It is worth noting that in \eqref{EQ3MZN} either of these two conditions may fail. In Example~\ref{EXA1},
\[
	\mu(dz)=\frac{\alpha}{2}\left(-|z|^{-1}dz|_{(-\infty,0)}+|z|^{-1}dz|_{(0,\infty)} \right)
\]
is not finite. Another example below shows the possibility of $|\mu(\{z\})|=1$. 

\begin{example}\label{EXA3}
For a constant $0<\alpha<1$, take 
\[
	\rho(x)=\left\lbrace
	\begin{aligned}
	&|x|^\alpha, \quad\quad\;\;\, x<0,\\
	&|x|^\alpha+1,\quad x\geq 0. 
	\end{aligned} \right.
\]
Mimicking Example~\ref{EXA1}, one can check that all the conditions in Corollary~\ref{THM5} still hold, and $X$ is irreducible and conservative. Particularly, $(X_t,\mathbf{P}_x)$ is a semi-martingale for any $x\in \bR$. However in \eqref{EQ3MZN} for this case, $\rho(0-)=0, \rho(0)=1$ and hence $\mu(\{0\})=(\rho(0)-\rho(0-))/(\rho(0)+\rho(0-))=1$.  
\end{example}


\section{SDEs involving symmetric semi-martingale local times}

The remainder of this paper is devoted to the study of the SDE of an unknown semi-martingale $Y=(Y_t)_{t\geq 0}$:
\begin{equation}\label{EQ3YTT}
\left\lbrace
\begin{aligned}
	&dY_t=dW_t+\int_\bR \mu(dz) dL^z_t(Y),\\
	&Y_0=x \in \mathbb{R},
\end{aligned}\right.
\end{equation}
where $(W_t)_{t\geq 0}$ is a standard Brownian motion and $L^z(Y)$ is the symmetric semi-martingale local time of $Y$ at $z$. Basically, we always impose that
\begin{itemize}
\item[(M0)] $\mu=\mu^+-\mu^-$ is the difference of two positive Borel measures on $\bR$ with $\mu^+\perp \mu^-$ and 
\begin{equation}\label{EQ5MZZ}
	|\mu(\{z\})|\leq 1,\quad \forall z\in \bR. 
\end{equation}
\end{itemize} 
Denote $\Xi^\pm:=\{z\in \bR: \mu(\{z\})=\pm 1\}$ and
\[
	\Xi:=\Xi^+\cup \Xi^-=\{z\in \bR: |\mu(\{z\})|=1\}.
\]
A point $z\in \Xi$ is usually called a barrier, as it turns out in \cite{BE14} that if $\mu$ is a Radon signed measure on $\bR$ then the solution $Y$ (if exists) cannot reach the left (resp. right) side of $z$ for the case $x\geq z\in \Xi^+$ (resp. $x\leq z\in \Xi^-$). Meanwhile, we call $z\in \Xi^+$ (resp. $z\in \Xi^-$) a right (resp. left) barrier.
On the other hand, the restriction \eqref{EQ5MZZ} seems necessary for the existence of solutions to \eqref{EQ3YTT}. In fact in certain situations if $|\mu(\{z\})|>1$ for some $z$, then \eqref{EQ3YTT} has no solutions as shown in \cite{LG84, BC05}. It is worth noting that when $|\mu|$ is finite, the weak solutions to \eqref{EQ5MZZ} exist and the pathwise uniqueness holds in a certain meaning (see \cite{LG84} for the case $|\mu(\{z\})|<1$ for all $z\in \bR$ and \cite{BC05} for the case admitting $|\mu(\{z\})|=1$).  

What we are mainly concerned with is the general skew Brownian motion related to the SDE \eqref{EQ3YTT}, whose definition is given as below. Particularly, Lemma~\ref{COR10} tells us if such a general skew Brownian motion $X$ exists, then $(X_t,\mathbf{P}_x)$ is a weak solution to \eqref{EQ3YTT} for all $x\in \cup_{k\geq 1}\tI_k$. If \eqref{EQ3YTT} is well posed further, then the unique solution coincides with $X$ and we can derive deeper descriptions about this solution by means of Dirichlet forms. On the other hand, if two different general skew Brownian motions are related to \eqref{EQ3YTT}, then the uniqueness of weak solutions (as well as the pathwise uniqueness) of \eqref{EQ3YTT} will not hold. 

\begin{definition}\label{DEF5}
Let $X$ be a general skew Brownian motion with the density function $\rho$. If all the conditions in Lemma~\ref{COR10} are satisfied and \eqref{EQ3MZN} holds for $\mu$ in \eqref{EQ3YTT}, we say $X$ is related to \eqref{EQ3YTT}. 
\end{definition}
\begin{remark}
By \eqref{EQ3MZN}, the condition \eqref{EQ5MZZ} is definitely necessary for the existence of general skew Brownian motions related to \eqref{EQ3YTT}. 
\end{remark}

In what follows, we will prepare a lemma in \S\ref{SEC51} and then present the main result Theorem~\ref{THM58} to characterize the existence of general skew Brownian motions related to \eqref{EQ3YTT}. 

\subsection{A lemma}\label{SEC51}

Write $|\mu|=\mu^++\mu^-$ and
\[
\mu^\pm=\mu^\pm_\mathrm{c}+\mu^\pm_\mathrm{d}
\]
where $\mu^\pm_\rc$ and $\mu^\pm_\rd$ are the discrete and continuous parts of $\mu^\pm$. Further set $\mu_\rc:=\mu^+_\rc-\mu^-_\rc$ and $\mu_\rd:=\mu^+_\rd-\mu^-_\rd$.
Moreover, $|\mu_\rc|:=\mu^+_\rc+\mu^-_\rc$ and $|\mu_\rd|:=\mu_\rd^++\mu_\rd^-$. For the sake of convenience, set
\[
	\mu_z:=\mu(\{z\}),\quad \mu^+_z:=\mu^+(\{z\}),\quad \mu^-_z:=\mu^-(\{z\}). 
\]

Let $I=(a,b)$ be an open interval and assume that  $|\mu|$ is Radon on it (i.e. for any compact set $K\subset I$, $|\mu|(K)<\infty$) and $\Xi\cap I=\emptyset$, i.e. $|\mu_z|<1$ for all $z\in I$. Take a fixed point $e\in I$ and define a function on $I$ as follows: 
\begin{equation}\label{EQ5VIZ}
	\varrho_I(z):=\left\lbrace
		\begin{aligned}
		&\exp\{2\mu_\rc((e,z])\}\prod_{e< y\leq z}\frac{1+\mu_y}{1-\mu_y},\quad e\leq z<b, \\
 &		\exp\{-2\mu_\rc((z,e])\}\prod_{z<y\leq e}\frac{1-\mu_y}{1+\mu_y},\quad a<z< e.
		\end{aligned}
		\right. 
\end{equation}
Replacing $\mu$ by $\mu^+$ or $\mu^-$ in the above expression, one can define the function $\varrho^+_I$ or $\varrho^-_I$. The following lemma concerning $\varrho_I$ as well as $\varrho^+_I$ and $\varrho^-_I$ is elementary but crucial to what follows. Note that in the sixth and seventh assertions, we take the canonical version of $\varrho_I$, i.e. $\varrho_I(a-):=0$ and $\varrho_I(b)=0$. Clearly, these values have no effects on the property that $\varrho_I$ is of bounded variation on $[a,e]$ or $[e,b]$. 

\begin{lemma}\label{LM58}
Assume that $|\mu|$ is Radon on $I=(a,b)$ and $|\mu_z|<1$ for all $z\in I$. Then the following hold:
\begin{itemize}
\item[(1)] $\varrho_I$ is a cadlag function locally of bounded variation on $I$.  
\item[(2)] $\varrho^+_I$ and $\varrho^-_I$ are cadlag and increasing.
\item[(3)] $\varrho_I(z)>0$ and $\varrho_I(z-)>0$, where $\varrho_I(z-)$ is the left limit of $\varrho$ at $z$, for all $z\in I$. Particularly, for any compact interval $K\subset I$, there exists a constant $c_K>1$ such that $1/c_K<\varrho_I(z)<c_K$ for all $z\in K$. Similar conclusions hold for $\varrho^+_I$ and $\varrho^-_I$. 
\item[(4)] $\varrho_I=\varrho^+_I/\varrho^-_I$. 
\item[(5)] $\varrho_I$ is the unique (up to a multiplicative constant) cadlag function locally of bounded variation on $I$ such that
\[
	\frac{\nu_{\varrho_I}(dz)}{\varrho_I(z)+\varrho_I(z-)}=\mu|_I(dz),
\]
where $\nu_{\varrho_I}$ is the induced Radon signed measure of $\varrho_I$ (see \S\ref{SEC42}). 
\item[(6)] If $\mu^-((a,e])<\infty$ (resp. $\mu^+([e,b))<\infty$), then $\varrho_I$ can be extended to a function of bounded variation on $[a,e]$ (resp. $[e,b]$). Particularly, the limit
\begin{equation}\label{EQ5VIA}
	\varrho_I(a):=\lim_{z\downarrow a}\varrho_I(z),\quad  (\text{resp. } \varrho_I(b-):=\lim_{z\uparrow b}\varrho_I(z))
\end{equation}
exists.
\item[(7)] $\varrho_I$ can be extended to a function of bounded variation on $[a,e]$ (resp. $[e,b]$) and the limit \eqref{EQ5VIA} is positive, if and only if $|\mu|((a,e])<\infty$ (resp. $|\mu|([e,b))<\infty$). 
\end{itemize}
\end{lemma}
\begin{proof}
Fix $z\in I$ with $z>e$ and write $\{y\in (e,z]: \mu_y\neq 0\}=:\{y_n: n\geq 1\}$. We first show 
\begin{equation}\label{EQ5EYZ}
	\prod_{e<y\leq z}\frac{1+\mu_y}{1-\mu_y}=\prod_{n\geq 1} \frac{1+\mu_{y_n}}{1-\mu_{y_n}}=\prod_{n\geq 1}\left(1+ \frac{2\mu_{y_n}}{1-\mu_{y_n}}\right)
\end{equation}
is absolutely convergent. Note that $\sum_{n\geq 1}|\mu_{y_n}|<\infty$. Thus for some $N\in \bN$, $|\mu_{y_n}|<1/2$ for all $n>N$. It follows that 
\[
\begin{aligned}
\sum_{n\geq 1} \left|\frac{2\mu_{y_n}}{1-\mu_{y_n}} \right|&=\sum_{1\leq n\leq N} \left|\frac{2\mu_{y_n}}{1-\mu_{y_n}} \right|+\sum_{n>N} \left|\frac{2\mu_{y_n}}{1-\mu_{y_n}} \right| \\
&\leq \sum_{1\leq n\leq N} \left|\frac{2\mu_{y_n}}{1-\mu_{y_n}} \right|+4\sum_{n>N}|\mu_{y_n}| \\
&<\infty. 
\end{aligned}\]
This leads to the absolute convergence of \eqref{EQ5EYZ}. Similar convergence holds also for $z<e$. Hence $\varrho_I$ is well defined. Moreover, one can easily find that $\varrho_I$ is cadlag and $\varrho_I(z)$, $\varrho_I(z-)>0$ for all $z\in I$. Meanwhile, $\varrho_I^\pm$ is well defined and cadlag, and $\varrho^\pm_I(z)$, $\varrho^\pm_I(z-)>0$ for all $z\in I$ by a similar derivation. On the other hand, $\varrho_I^\pm$ is increasing and the forth assertion $\varrho_I=\varrho^+_I/\varrho^-_I$ is obvious. Particularly, $\varrho^\pm_I$ is locally of bounded variation on $I$ and thus so is $\varrho_I$, since $1/\varrho^-_I$ is bounded on any compact subinterval of $I$. The first four assertions are concluded. The fifth one can be deduced by a straightforward computation, see also \cite[Lemma~2.1]{LG84}.  

For the sixth assertion, we only consider the case concerning $a$. Another one can be proved similarly. Suppose $\mu^-((a,e])<\infty$. Then $$\prod_{a<y\leq e}\frac{1-\mu^-_y}{1+\mu^-_y}$$ is absolutely convergent and $\exp\{\mu_\rc^-((a,e])\}<\infty$. Thus $$0<\varrho^-_I(a):=\lim_{z\downarrow a}\varrho^-_I(z)<\infty.$$ Particularly, $1/\varrho^-_I$ is bounded and decreasing on $[a,e]$. If $\mu^+((a,e])<\infty$, one can also deduce that $0<\varrho^+_I(a):=\lim_{z\downarrow a}\varrho^+_I(z)<\infty$ and $\varrho^+_I$ is of bounded variation on $[a,e]$. This implies that \eqref{EQ5VIA} exists, $\varrho_I(a)=\varrho_I^+(a)/\varrho_I^-(a)>0$ and $\varrho_I$ is of bounded variation on $[a,e]$. If $\mu^+((a,e])=\infty$, then $\mu^+_\rc((a,e])=\infty$ or $\mu^+_\rd((a,e])=\infty$. The former case leads to $$\lim_{z\downarrow a}\exp\{-2\mu^+_\rc((z,e])\}=0$$ and the latter one implies 
\[
\lim_{z\downarrow a}\prod_{z<y\leq e}\frac{1-\mu^+_y}{1+\mu^+_y}=0. 
\]
Hence $\varrho^+_I(a)=0$ and $\varrho^+_I$ is still of bounded variation on $[a,e]$. Therefore, $\varrho_I=\varrho^+_I/\varrho^-_I$ is of bounded variation on $[a,e]$. 

The sufficiency of the last assertion is already indicated in the proof of the previous assertion. It suffices to show the necessity of $|\mu|((a,e])<\infty$. In fact, it follows from $\varrho_I(a)>0$ and the third assertion that for some constant $\delta>0$, $\varrho_I(z)\geq \delta$ and $\varrho_I(z-)\geq \delta$ for all $z\in (a,e]$. Note that for $a<z\leq e$, 
\[
	|\nu_{\varrho_I}|(\{z\})=|\varrho_I(z)-\varrho_I(z-)|=\left|\frac{2\mu_z}{1-\mu_z}\right|\cdot \varrho_I(z-)\geq |\mu_z|\delta.
\]
Since $\varrho_I$ is of bounded variation on $[a,e]$, we can conclude $\sum_{a<z\leq e}|\mu_z|<\infty$. Hence $|\mu_\rd|((a,e])<\infty$. This also implies that $\prod_{a<y\leq e}\frac{1-\mu_y}{1+\mu_y}$ is absolutely convergent and then 
\[
	z\mapsto F(z):=\varrho_I(z)\cdot \prod_{z<y\leq e}\frac{1+\mu_y}{1-\mu_y}
\]	
is of bounded variation on $[a,e]$. Clearly, $F(z)=\exp\{-2\mu_\rc((z,e]) \}$ for $z\in (a,e)$ and there exists a constant $\tilde{\delta}>1$ such that $1/\tilde{\delta}\leq F(z)\leq \tilde{\delta}$ for all $z\in [a,e]$. Then we have $\log F$ is of bounded variation on $[a,e]$. Therefore, $|\mu_\rc|((a,e])<\infty$. 
\end{proof}

\subsection{Existence of related general skew Brownian motions}


Now we move on to state the main result. 
Denote 
\[
	G:=\{z\in \bR: \exists \varepsilon>0, |\mu|((z-\varepsilon,z+\varepsilon))<\infty\}. 
\]	
Clearly, $G$ is open and can be written as a union of disjoint open intervals: 
\begin{equation}\label{EQ5UNI}
	G=\cup_{n\geq 1}I_n=\cup_{n\geq 1}(a_n,b_n).
\end{equation}
Then $|\mu|$ is Radon on $I_n=(a_n,b_n)$. Take a fixed point $e_n\in I_n$. When $I_n\cap \Xi=\emptyset$, denote the function $\varrho_{I_n}$ in \eqref{EQ5VIZ} with $I=I_n$ and $e=e_n$ by $\varrho_n$. Define a function $\varrho$ on $\cup_{n\geq 1}I_n$ by $\varrho:=\varrho_n$ on each $I_n$. The main result of this section is as follows. Note that the first condition thereof indicates $I_n\cap \Xi=\emptyset$ for all $n$ and $\varrho$ in the latter conditions is given as above.  

\begin{theorem}\label{THM58}
Assume (M0) and let $G$ be in \eqref{EQ5UNI}. Then there exists a general skew Brownian motion related to \eqref{EQ3YTT}, if and only if the following hold:
\begin{itemize}
\item[(1)] $\Xi^+=\{a_n: a_n>-\infty, |\mu|((a_n,e_n))<\infty, n\geq 1\}$ and $\Xi^-=\{b_n:b_n<\infty, |\mu|((e_n,b_n))<\infty, n\geq 1\}$;  
\item[(2)] $|\mu|(G^c\setminus \Xi)=0$ (and particularly $G^c$ is nowhere dense);
\item[(3)] When $a_n>-\infty$ (resp. $b_n<\infty$), $\int_{a_n}^{e_n}\varrho(y)dy<\infty$ (resp. $\int_{e_n}^{b_n}\varrho(y)dy<\infty$); 
\item[(4)] When $a_n>-\infty$ and $\int_{a_n}^{e_n}\frac{dy}{\varrho(y)}<\infty$ (resp. $b_n<\infty$ and $\int_{e_n}^{b_n}\frac{dy}{\varrho(y)}<\infty$), $\varrho|_{(a_n,e_n]}$ (resp. $\varrho|_{[e_n,b_n)}$) can be extended to a function of bounded variation on $[a_n,e_n]$ (resp. $[e_n,b_n]$);
\item[(5)] If a constant $L>0$ exists such that $(L,\infty)\cap G^c=\emptyset$ (resp. $(-\infty, -L)\cap G^c=\emptyset$), then 
\begin{equation}\label{EQ5LXV2}
	\int_L^\infty\frac{dx}{\varrho(x)}\int_L^x \varrho(y)dy=\infty,\quad \left(\text{resp. }\int_{-\infty}^{-L} \frac{dx}{\varrho(x)}\int_x^{-L} \varrho(y)dy=\infty \right).
\end{equation}
\end{itemize}
\end{theorem}
\begin{remark}
The last condition is used only for guaranteeing the conservativeness of the expected general skew Brownian motions. 
\end{remark}

Before the proof, we show some facts concerning the fourth condition in Theorem~\ref{THM58}. Lemma~\ref{LM58} tells us when $\mu^-((a_n,e_n])<\infty$, $\varrho|_{(a_n,e_n]}$ can be extended to a function of bounded variation on $[a_n,e_n]$. In the meanwhile, $\varrho(a_n)=\lim_{z\downarrow a_n}\varrho(z)=0$ whenever $\mu^+((a_n,e_n])=\infty$ and $\varrho(a_n)=\lim_{z\downarrow a_n}\varrho(z)>0$ whenever $\mu^+((a_n,e_n])<\infty$. For the case $\mu^-((a_n,e_n])=\infty$ and $\mu^+((a_n,e_n])<\infty$, one can easily conclude that $\lim_{z\downarrow a_n}\varrho(z)$ diverges to $\infty$ and thus $\varrho|_{(a_n,e_n]}$ cannot be extended to a function of bounded variation on $[a_n,e_n]$. Finally when $\mu^-((a_n,e_n])=\infty$ and $\mu^+((a_n,e_n])=\infty$, it is possible to find some examples where $\rho|_{(a_n,e_n]}$ can or cannot be extended to a function of bounded variation on $[a_n,e_n]$, see Example~\ref{EXA14}. Once the former situation occurs, we must have $\varrho(a_n)=\lim_{z\downarrow a_n}\varrho(z)=0$.  

\subsection{Proof of Theorem~\ref{THM58}}\label{SEC6}

\subsubsection{Necessity}

The lemma below is very useful for proving the necessity of Theorem~\ref{THM58}. 

\begin{lemma}\label{LM11}
Assume (M0) and there is a general skew Brownian motion $X$ with the density function $\rho$ related to \eqref{EQ3YTT}. Let $\{\tI_k=\langle \fa_k,\fb_k\rangle:k\geq 1\}$ be the set of effective intervals of $X$. Then the following hold:
\begin{itemize}
\item[(1)] If $I=(a,b)$ is an open interval such that $|\mu|(I)<\infty$, then $I\subset \tI_k$ for some $k$ and $I\cap \Xi=\emptyset$. 
\item[(2)] $z\in G$ if and only if $z\in \cup_{k\geq 1}\tI_k$ and $\rho(z),\rho(z-)>0$. 
\item[(3)] For any $z\in \Xi^+$ (resp. $z\in \Xi^-$), there exists a constant $\varepsilon>0$ such that 
\[
	|\mu|((z,z+\varepsilon))<\infty, \quad (\text{resp. } |\mu|((z-\varepsilon,z))<\infty)
\]
and 
\[
	(z,z+\varepsilon)\cap \Xi=\emptyset,\quad (\text{resp. } (z-\varepsilon, z)\cap \Xi=\emptyset).
\]
\end{itemize} 
\end{lemma} 
\begin{proof}
\begin{itemize}
\item[(1)] We first show $I\subset \tI_k$ for some $k$. 
Arguing by contradiction, suppose $\fa_1<\fa_2$ and $J_1:=I\cap (\fa_1,\fb_1), J_2:=I\cap (\fa_2,\fb_2)$ are not empty. Write $J_1=(\fc_1,\fd_1), J_2=(\fc_2,\fd_2)$. Then $\fd_1=\fb_1$ and $\fc_2=\fa_2$. Note that $J_i\cap \Xi=\emptyset$. Indeed, if $z\in J_i\cap \Xi$, then there exists a constant $\varepsilon>0$ such that $(z-\varepsilon,z+\varepsilon)\subset J_i$ and $(z-\varepsilon,z+\varepsilon)\cap \Xi=\{z\}$, since otherwise we have $|\mu|(J_i)=\infty$ leading to contradiction. Hence $\rho=c_1\varrho_{(z-\varepsilon,z)}$ on $(z-\varepsilon,z)$ and $\rho=c_2\varrho_{(z,z+\varepsilon)}$ on $(z,z+\varepsilon)$ for two constants $c_1$ and $c_2$ by Lemma~\ref{LM58}~(5). But $|\mu|((z-\varepsilon,z+\varepsilon))<\infty$ indicates $\rho(z-)>0$ and $\rho(z)>0$ by Lemma~\ref{LM58}~(7). Thus $|\mu(\{z\})|<1$ which contradicts with $z\in \Xi$. By applying Lemma~\ref{LM58}~(5) and (7) to $J_i$, it follows from $J_i\cap \Xi=\emptyset$ and $|\mu|(J_i)<\infty$ that $\rho$ is of bounded variation on $[\fe_1,\fb_1]$ (resp. $[\fa_2, \fe_2]$) with $\rho(\fb_1-)>0$ (resp. $\rho(\fa_2)>0$). As a result, $\int_{\fe_1}^{\fb_1}\frac{dy}{\rho(y)}+\int_{\fa_2}^{\fe_2}\frac{dy}{\rho(y)}<\infty$, which implies $\fb_1\in \tI_1$, $\mu(\{\fb_1\})=-1$ and $\fa_2\in \tI_2$, $\mu(\{\fa_1\})=1$. Note that $\tI_1\cap \tI_2=\emptyset$ tells us $\fb_1<\fa_2$. Furthermore, we can obtain that any $\tI_k\subset [\fb_1,\fa_2]$ is closed and $\mu(\{\fa_k\}=1, \mu(\{\fb_k\})=-1$ by mimicking the argument above. Particularly, there are infinite effective intervals between $\fb_1$ and $\fa_2$. This leads to the contradiction $|\mu|(I)=\infty$. Eventually we can conclude that $I\subset \tI_k$ for some $k$. Next, by the same argument for proving $J_i\cap \Xi=\emptyset$, one can also deduce that $I\cap \Xi=\emptyset$.
\item[(2)] The necessity is clear since the first assertion implies $z\in \tI_k$ for some $k$ and $z\notin \Xi$. To the sufficiency, note that $z\in (\fa_k,\fb_k)$ for some $k$ since $\rho(z),\rho(z-)>0$. Then there exist two constants $\varepsilon>0$ and $\delta_\varepsilon>0$ such that $\rho$ is of bounded variation on $[z-\varepsilon,z+\varepsilon]$ ($\subset (\fa_k,\fb_k)$) and $\rho(y),\rho(y-)\geq \delta_\varepsilon$ for all $y\in (z-\varepsilon,z+\varepsilon)$. From Lemma~\ref{COR10} we obtain
\[
	|\mu|((z-\varepsilon,z+\varepsilon))= \int_{(z-\varepsilon,z+\varepsilon)} \frac{|\nu_\rho|(y)}{\rho(y)+\rho(y-)}\leq \frac{1}{2\delta_\varepsilon} |\nu_\rho|((z-\varepsilon,z+\varepsilon))<\infty.
\]
\item[(3)] For $z\in \Xi^+$, we have $\rho(z)>0$ and $\rho(z-)=0$. Mimicking the argument for the second assertion, one can obtain that $|\mu|((z,z+\varepsilon))<\infty$ for some constant $\varepsilon>0$. Then $(z,z+\varepsilon)\cap \Xi=\emptyset$ is implied by the first assertion. 
\end{itemize}  
That completes the proof. 
\end{proof} 

Now we have a position to prove the necessity of Theorem~\ref{THM58}. 

\begin{proof}[Proof of the necessity of Theorem~\ref{THM58}]
Let $X$ be a general skew Brownian motion with the density function $\rho$ related to \eqref{EQ3YTT}. The set of its effective intervals is denoted by $\{\tI_k=\langle \fa_k,\fb_k\rangle:k\geq 1\}$. The first assertion of Lemma~\ref{LM11} particularly indicates $I_n\subset \tI_k$ for some $k$ and $I_n\cap \Xi=\emptyset$ for all $I_n$ in \eqref{EQ5UNI}. Then by applying Lemma~\ref{LM58}~(5), there is a constant $c_n>0$ such that $\rho=c_n\varrho$ on $I_n$. 
\begin{itemize}
\item[(1)] Clearly	 $\Xi\subset G^c$. Take $z\in \Xi^+$. By Lemma~\ref{LM11}, a constant $\varepsilon>0$ exists such that $(z,z+\varepsilon)\subset G$ but $z\notin G$. We have $z=a_n$ for some $n$ and $|\mu|((a_n,e_n))<\infty$. To the contrary, let $a_n>-\infty$ with $|\mu|((a_n,e_n))<\infty$. Arguing by contradiction, suppose $a_n\notin \Xi^+$. Lemma~\ref{LM58}~(7) tells us $\rho$ is of bounded variation on $[a_n,e_n]$ and $\rho(a_n)>0$. Particularly, $[a_n,e_n]\subset \tI_k$. Since $a_n\notin \Xi^+$, we must have $\rho(a_n-)>0$ and $\rho$ is of bounded variation on $[a_n-\tilde{\varepsilon},a_n+\tilde{\varepsilon}]\subset \tI_k$ for a constant $\tilde{\varepsilon}>0$. By taking a smaller constant $\epsilon>0$, one can obtain that $\rho(z),\rho(z-)>\frac{1}{2}(\rho(a_n)\wedge \rho(a_n-))$ for any $z\in (a_n-\epsilon,a_n+\epsilon)$.  It follows that 
\[
\begin{aligned}
	|\mu|((a_n-\epsilon,a_n+\epsilon))&= \int_{(a_n-\epsilon,a_n+\epsilon)} \frac{|\nu_\rho|(y)}{\rho(y)+\rho(y-)} \\ 
	&\leq \frac{1}{\rho(a_n)\wedge \rho(a_n-)} |\nu_\rho|((a_n-\epsilon,a_n+\epsilon))<\infty,
\end{aligned}
\]
 which contradicts with $a_n\notin G$. The expression of $\Xi^-$ can be deduced similarly.
\item[(2)] Set $N:=\left(\cup_{k\geq 1}\tI_k\right)^c$ and $Z_\rho:=\{z\in \cup_{k\geq 1}\tI_k: \rho(z)=\rho(z-)=0\}$. Note that $G^c=\Xi\cup Z_\rho\cup N$ and $|\mu|(Z_\rho\cup N)=0$ by Lemma~\ref{COR10}. Thus we have $|\mu|(G^c\setminus \Xi)\leq |\mu|(Z_\rho\cup N)=0$. This implies $G^c$ is nowhere dense. Indeed, if $J\subset G^c$ is a non-empty open interval, then $J\cap \Xi=\emptyset$ by the first assertion and hence $|\mu|(J)=0$. As a result, $J\subset G$ by the definition of $G$, which leads to contradiction. 
\item[(3)] When $a_n>-\infty$, we have $\int_{a_n}^{e_n}\rho(y)dy<\infty$ since $\rho\in L^1_\mathrm{loc}(\bR)$.  It suffices to note that $\rho=c_n\varrho$ on $I_n$. 
\item[(4)] Let $a_n$ be such an endpoint. Then $\int_{a_n}^{e_n}\frac{dy}{\rho(y)}<\infty$ implies $[a_n,e_n]\subset \tI_k$. Particularly, $\rho$ is of bounded variation on $[a_n,e_n]$. Since $\varrho_n=\rho/c_n$ on $I_n$,  it follows that $\varrho_n$ can be extended to a function of bounded variation on $[a_n,e_n]$. 
\item[(5)] When such $L$ exists, it holds $(L,\infty)\subset I_n\subset \tI_k$ or $(-\infty, -L)\subset I_n\subset \tI_k$ for some $n,k$. Then \eqref{EQ5LXV2} is a consequence of the conservativeness of $X$ by Proposition~\ref{THM4} and $\rho=c_n\varrho$ on $I_n$. 
\end{itemize}
That completes the proof.
\end{proof}
\begin{remark}\label{RM67}
Let us summarize the description of $\rho$. It is already indicated that $\rho=c_n\varrho$ on $I_n$. Note that $G^c=\Xi\cup Z_\rho\cup N$ where $N=\left(\cup_{k\geq 1}\tI_k\right)^c$ and $Z_\rho=\{z\in \cup_{k\geq 1}\tI_k: \rho(z)=\rho(z-)=0\}$. The $\fm$-polar set $N$ makes no sense and without loss of generality, one can impose $\rho(z)=0$ for $z\in N$. Every point $z$ in $\Xi^+$ (resp. $\Xi^-$) is the left (right) endpoint of some $I_n$ and it must hold 
\[
	\rho(z)=\lim_{y\downarrow a_n}c_n\varrho(y)>0, \quad \left(\rho(z-)=\lim_{y\uparrow b_n}c_n\varrho(y)>0 \right)
\] 
and $\rho(z-)=0$ (resp. $\rho(z)=0$). Finally for any $z\in Z_\rho$, we have $\rho(z)=\rho(z-)=0$. 
\end{remark}

\subsubsection{Sufficiency}\label{SEC632}

From now on we assume the five conditions in Theorem~\ref{THM58} are satisfied. 
To show the existence of related general skew Brownian motions, it suffices to construct a density function $\rho$ satisfying all the conditions in Definition~\ref{DEF5}. Note that $\rho$ must be given as follows by Remark~\ref{RM67}:
\begin{description}
\item[(i)] For all $n\geq 1$, there is a constant $c_n>0$ such that $\rho=c_n\varrho$ on $I_n$;
\item[(ii)] For $a_n\in \Xi^+$, $\rho(a_n):=\lim_{z\downarrow a_n}c_n\varrho(z)$ (the existence of this limit is guaranteed by Lemma~\ref{LM58}~(7) and the first condition);
\item[(iii)] $\rho(z):=0$ for $z\in G^c\setminus \Xi^+$.
\end{description}
So the task is to find a suitable set of positive constants $\{c_n: n\geq 1\}$. To this end, we analyse all conditions required by Definition~\ref{DEF5} as follows.

(a) \emph{The basic assumption (A) in \S\ref{SEC3}}. For any bounded $I_n$, denote $A_n:=\int_{I_n}\varrho(y)dy$. It follows from the third condition that $A_n$ is finite. Then the following lemma characterizes this assumption. 

\begin{lemma}
The assumption (A) holds, if and only if for any $L>0$, 
\begin{equation}\label{EQ5NINL}
	\sum_{n: I_n\subset (-L, L)} c_nA_n<\infty. 
\end{equation}
\end{lemma}
\begin{proof}
Note that \eqref{EQ5NINL} is equivalent to $\rho\in L^1_\mathrm{loc}(\bR)$. Hence the necessity is clear. To prove the sufficiency, since $\rho(z)=c_n\varrho(z)>0$ for $z\in I_n$ and $G^c$ is nowhere dense by the second condition, it follows that $\int_U\rho(z)dz>0$ for any non-empty open set $U\subset \bR$, which leads to \eqref{EQ2RLR}. Moreover, one can obtain $G\subset S(\rho)^c$ by Lemma~\ref{LM58}~(3) and thus $S(\rho)\subset G^c$. But $\rho\equiv 0$ on $G^c\setminus \Xi^+$ and $\Xi^+$ is countable by the first condition. We have $\rho=0$ a.e. on $S(\rho)\subset G^c$. Eventually (A) holds. That completes the proof. 
\end{proof}

(b) \emph{Effective intervals $\{(\tI_k,\fs_k):k\geq 1\}$ of $X$}. No extra conditions are needed for the existence of effective intervals. Instead, we should derive their expression by means of $\{I_n: n\geq 1\}$ and $\{c_n: n\geq 1\}$. To accomplish it, we repeat a conception introduced in \cite{LY172}. 

\begin{definition}
Given the density function $\rho$, we say $I_i$ and $I_j$ are scale-connected or $I_i$ is scale-connected to $I_j$, if provided that $\int_{(e_i, e_j)}\frac{dy}{\rho(y)}<\infty$, where $(e_i, e_j)$ is the open interval ended by $e_i$ and $e_j$ no matter which is bigger. 
\end{definition}
\begin{remark}
Denote $B_n:=\int_{I_n}\frac{dy}{\varrho(y)}$, $B^\rr_n:=\int_{e_n}^{b_n}\frac{dy}{\varrho(y)}$ and $B^\rl_n:=\int_{a_n}^{e_n}\frac{dy}{\varrho(y)}$ for all $n$. Assume without loss of generality $e_i<e_j$. The scale-connection between $I_i$ and $I_j$ means $(e_i,e_j)\cap G^c$ is of zero Lebesgue measure, and 
\begin{equation}\label{EQ5NEE}
	\sum_{n: I_n\subset (e_i,e_j)} \frac{B_n}{c_n}+\frac{B_i^\rr}{c_i}+\frac{B_j^\rl}{c_j}<\infty. 
\end{equation}
\end{remark}

Roughly speaking, $\{\tI_k:k\geq 1\}$ is obtained by gluing all scale-connected intervals in $\{I_n: n\geq 1\}$ (cf. \cite[\S3.3]{LY172}). The details are explained as follows. The scale-connection is an equivalent relation for the intervals in $\{I_n:n\geq 1\}$. Denote all its induced equivalence classes by $\{\fI_k:k\geq 1\}$, where $\fI_k\subset \{I_n: n\geq 1\}$ contains mutually scale-connected intervals. Note that if $I_i$ is scale-connected to $I_j$, then all $I_n$ located between $I_i$ and $I_j$ must be scale-connected to them. Hence $\fJ_k$ looks like a ``connected" cluster of intervals. Set 
\[
	\fa_k:=\inf\{x\in J: J\in \fI_k\},\quad \fb_k:=\sup\{x\in J: J\in \fI_k\}. 
\]
Take a fixed point $\fe_k\in (\fa_k,\fb_k)$ and write 
\begin{equation}\label{EQ5SKX}
	\fs_k(x):=\int_{\fe_k}^x\frac{dx}{\rho(x)},\quad x\in (\fa_k,\fb_k).
\end{equation}
Then $\fs_k$ is absolutely continuous and strictly increasing on $(\fa_k,\fb_k)$. Define 
\begin{equation}\label{EQ5IKA}
	\tI_k:=\langle \fa_k,\fb_k\rangle,
\end{equation}
where $\fa_k\in \tI_k$ (resp. $\fb_k\in \tI_k$) if and only if $\fa_k+\lim_{x\downarrow \fa_k}\fs_k(x)>-\infty$ (resp. $\fb_k+\lim_{x\uparrow \fb_k}\fs_k(x)<\infty$. Eventually we obtain the following. 

\begin{lemma}
Let $\fs_k$ and $\tI_k$ be in \eqref{EQ5SKX} and \eqref{EQ5IKA}. Then $\{(\tI_k,\fs_k): k\geq 1\}$ is the set of effective intervals. 
\end{lemma}
\begin{proof}
It is easy to verify that $\tI_k$ are mutually disjoint and $\fs_k$ is adapted to $\tI_k$ (see the definition of adaptedness in \S\ref{SEC21}). Then Theorem~\ref{THM2} tells us there is a Dirichlet form $(\tilde{\EE},\tilde{\FF})$ whose effective intervals are $\{(\tI_k,\fs_k):k\geq 1\}$. It suffices to show $(\EE,\FF)=(\tilde{\EE},\tilde{\FF})$. Indeed, by the representation of $(\tilde{\EE},\tilde{\FF})$ in Theorem~\ref{THM2}, one can deduce that $C_c^\infty(\bR)\subset \tilde{\FF}$ and $\tilde{\EE}(f,g)=\frac{1}{2}\int_\bR f'(x)g'(x)\rho(x)dx$ for $f,g\in C_c^\infty(\bR)$. In other words, $\tilde{\EE}(f,g)=\EE(f,g)$. Then applying \cite[Theorem~3.7]{LY172} to $(\tilde{\EE},\tilde{\FF})$, one can conclude that $C_c^\infty(\bR)$ is a core of $(\tilde{\EE},\tilde{\FF})$. Since $(\EE,\FF)$ is obtained by the closure of \eqref{EQ2EFG}, we have $(\EE,\FF)=(\tilde{\EE},\tilde{\FF})$. That completes the proof. 
\end{proof}

(c) \emph{$\rho$ is locally of bounded variation on each $\tI_k$}. Consider $I_n=(a_n,b_n)\subset \tI_k$ for some $k$. When $a_n\in\tI_k$ (resp. $b_n\in \tI_k$), it follows from the forth condition that $\varrho_n$ can be extended to a function on $[a_n,b_n)$ (resp. $(a_n,b_n]$), which is canonical and of bounded variation on $[a_n,e_n]$ (resp. $[e_n,b_n]$). Denote the closure of $I_n$ in $\tI_k$ by 
\[
	\hat{I}_n:=\langle a_n,b_n\rangle,
\]
i.e. $a_n\in \hat{I}_n$ or $b_n\in \hat{I}_n$ if $a_n\in\tI_k$ or $b_n\in \tI_k$, and the extension of $\varrho_n$ to $\hat{I}_n$ still by $\varrho_n$. Moreover, set $\varrho_n(a_n-):=0$ (resp. $\varrho_n(b_n):=0$) whenever  $a_n\in \hat{I}_n$ (resp. $b_n\in \hat{I}_n$). Further denote the induced Radon signed measure of $\varrho_n$ on $\hat{I}_n$ by $\nu_{\varrho_n}$ (see the explanation below Definition~\ref{DEF1}), and let $V_n:=|\nu_{\varrho_n}|(\hat{I}_n)$ be the total variation of $\nu_{\varrho_n}$. Note that if $\hat{I}_n=[a_n,b_n]$, then $V_n$ is finite and $\nu_{\varrho_n}(\hat{I}_n)=0$. 

\begin{lemma}\label{LM612}
\begin{itemize}
\item[(1)] For any $n\neq m$, $\nu_{\varrho_n}\perp \nu_{\varrho_m}$. 
\item[(2)] $\rho$ is locally of bounded variation on each $\tI_k$, if and only if 
for any $k$ and any compact interval $K\subset \tI_k$, 
\begin{equation}\label{EQ5NIN2}
\sum_{n: I_n\subset K}c_nV_n<\infty.
\end{equation}
In the meanwhile, the induced measure $\nu_\rho$ of $\rho$ is 
\begin{equation}\label{EQ5NNC}
	\nu_\rho=\sum_{n\geq 1}c_n\nu_{\varrho_n}. 
\end{equation}
\end{itemize}
\end{lemma}
\begin{proof}
For the first assertion, argue by contradiction and suppose $e_n<e_m$ but $\nu_{\varrho_n}$ is not singular to $\nu_{\varrho_m}$. This implies $b_n=a_m\in \hat{I}_n\cap \hat{I}_m$, $\nu_{\varrho_n}(\{b_n\})\neq 0$ and $\nu_{\varrho_m}(\{a_m\})\neq 0$. Since $\nu_{\varrho_n}(\{b_n\})\neq 0$ indicates $\varrho_n(b_n-)>0$, which is equivalent to $|\mu|([e_n,b_n))<\infty$ by Lemma~\ref{LM58}~(7), it follows from the first condition that $b_n\in \Xi^-$. Analogically, one can obtain $a_m\in \Xi^+$. Hence $b_n=a_m\in \Xi^+\cap \Xi^-=\emptyset$ which leads to contradiction. 

To prove the necessity of \eqref{EQ5NIN2}, it suffices to show $\nu_K:=\sum_{n: I_n\subset K}\nu_{\varrho_n}$ is a finite signed measure. Take $I_n\subset K$. Firstly, note that $\hat{I}_n=[a_n,b_n]$ and $\rho=c_n\varrho_n$ on $I_n$. It follows that $\nu_\rho=c_n\nu_{\varrho_n}=\nu_K$ on $I_n$. Next, we claim $|\nu_K(\{a_n\})|\leq |\nu_\rho(\{a_n\})|$. Without loss of generality, assume $|\nu_K(\{a_n\})|>0$. When $\nu_{\varrho_n}(\{a_n\})\neq 0$, we have $|\nu_K(\{a_n\})|=|\nu_{\varrho_n}(\{a_n\})|=c_n\varrho_n(a_n)>0$ by the first assertion. In the meanwhile, $a_n\in \Xi^+$ and $\rho(a_n)=c_n\varrho(a_n)$. Moreover, it holds that $a_n=\fa_k\in \tI_k$, or $a_n=b_m$ for some $m$, or $b_{p_i}\uparrow a_n$ for a subsequence $\{b_{p_i}:i\geq 1\}\subset \{b_n: b_n\in \tI_k\}$. In the first case, $|\nu_\rho(\{a_n\})|=\rho(a_n)=c_n\varrho_n(a_n)=|\nu_{\varrho_n}(\{a_n\})|$.  In the second case, since $\rho$ is of bounded variation on $K$, we have $\rho(a_n-)=\lim_{z\uparrow a_n}c_m\varrho_m(z)$ exists. Applying Lemma~\ref{LM58}~(7), $\rho(a_n-)$ must be equal to $0$ since $a_n\notin \Xi^-$. Hence $|\nu_\rho(\{a_n\})|=|\nu_{\varrho_n}(\{a_n\})|$ remains true. In the last case, note that $\rho(b_{p_i})=0$ and hence $\rho(a_n-)=\lim_{i\uparrow \infty}\rho(b_{p_i})=0$. We still have $|\nu_\rho(\{a_n\})|=|\nu_{\varrho_n}(\{a_n\})|$. When $\nu_{\varrho_n}(\{a_n\})= 0$, we have $a_n=b_m$ for some $I_m\subset K$ and $\nu_{\varrho_m}(\{b_m\})\neq 0$, i.e. $b_m\in \Xi^-$. One can obtain that $\rho(a_n)=c_n\varrho_n(a_n)=0$ and $\rho(a_n-)=c_m\varrho_m(b_m-)>0$. Hence $|\nu_\rho(\{a_n\})|=c_m|\varrho_m(b_m-)|=|\nu_{K}(\{a_n\})|$. Eventually we can conclude $|\nu_K(\{a_n\})|\leq |\nu_\rho(\{a_n\})|$. Similarly, $|\nu_K(\{b_n\})|\leq |\nu_\rho(\{b_n\})|$ holds. Finally, $|\nu_K|(K)\leq |\nu_K|(\cup_{n\geq 1}I_n)+|\nu_K|(\cup_{n\geq 1}\{a_n,b_n\})\leq |\nu_\rho|(\cup_{n\geq 1}I_n)+|\nu_\rho|(\cup_{n\geq 1}\{a_n,b_n\}) \leq  |\nu_\rho|(K)<\infty$. 

To prove the sufficiency of \eqref{EQ5NIN2}, fix $k\geq 1$ and take $e_p\in \tI_k$. Denote the right hand side of \eqref{EQ5NNC} by $\nu$. Then \eqref{EQ5NIN2} and the first assertion indicate that $\nu$ is a Radon signed measure on $\tI_k$. Define
\[
	\tilde{\rho}(z):=\left\lbrace
	\begin{aligned}
	&c_p\varrho_p(e_p)+\nu((e_p,z]),\quad e_p\leq z\in \tI_k,\\
	&c_p\varrho_p(e_p)-\nu((z,e_p]),\quad e_p>z\in \tI_k. 
	\end{aligned}
	\right. 
\] 
Clearly $\tilde{\rho}$ is a canonical function locally of bounded variation on $\tI_k$. Let us show $\rho=\tilde{\rho}$ on $\tI_k$, which leads to $\nu_\rho=\nu$. In fact, recall that $\nu_{\varrho_n}([a_n,b_n])=\varrho_n(b_n)-\varrho_n(a_n-)=0$ if $[a_n,b_n]\subset \tI_k$. For any $z\geq e_p$ and $z\in I_n\subset \tI_k$, we have
\[
	\tilde{\rho}(z)=c_p\varrho_p(e_p)+c_p\nu_{\varrho_p}((e_p,b_p])+c_n\nu_{\varrho_n}([a_n,z])=c_n\varrho_n(z)=\rho(z). 
\]
The case $z<e_p$ can be treated similarly. Hence $\rho=\tilde{\rho}$ on all $I_n$. For $z\in \tI_k\setminus \cup_{n\geq 1}[a_n,b_n]$, we can obtain $\tilde{\rho}(z)=0=\rho(z)$. Finally, it suffices to prove $\rho(a_n)=\tilde{\rho}(a_n)$ and $\rho(b_n)=\tilde{\rho}(b_n)$ for $a_n, b_n\in \tI_k$. We only treat the case $b_n\in \tI_k$ with $b_n>e_p$. When $b_n=a_m$ for some $m$, $\tilde{\rho}(b_n)=c_m\nu_{\varrho_m}(\{a_m\})=c_m\varrho_m(a_m)$. Note that $\tilde{\rho}(b_n)>0$ if and only if $b_n=a_m\in \Xi^+$. In this case, $\rho(b_n)=\rho(a_m)=\lim_{z\downarrow a_m}\rho(z)=c_m\varrho_m(a_m)=\tilde{\rho}(b_n)$. Otherwise if $\tilde{\rho}(b_n)=0$, then $b_n\in G^c\setminus \Xi^+$. Thus $\rho(b_n)=0=\tilde{\rho}(b_n)$. When no such $m$ exists, it follows from the first condition that $b_n\notin \Xi^+$ and hence $\rho(b_n)=0$. From the definition of $\tilde{\rho}$, one can also obtain $\tilde{\rho}(b_n)=0$. That completes the proof. 
\end{proof}

(d) \emph{Relation between $\mu$ and $\rho$}. It follows from $\rho=c_n\varrho_n$ on $I_n$ and Lemma~\ref{LM58}~(5) that 
\[
	\frac{\nu_\rho(dz)}{\rho(z)+\rho(z-)}=\frac{c_n\nu_{\rho_n}(dz)}{c_n(\varrho_n(z)+\varrho_n(z-))}=\mu(dz),\quad z\in I_n.
\]
Note that $|\nu_\rho|\left((\cup_{n\geq 1}[a_n,b_n])^c\right)=0$ by \eqref{EQ5NNC},  while $(\cup_{n\geq 1}[a_n,b_n])^c\subset G^c\setminus \Xi$ and the second condition indicates $|\mu|\left((\cup_{n\geq 1}[a_n,b_n])^c\right)=0$. 
Furthermore, consider the endpoint $a_n$. By mimicking the proof of the first assertion in Lemma~\ref{LM612}, we know that $\nu_\rho(\{a_n\})\neq 0$ is equivalent to $a_n\in \Xi^+$ or $a_n=b_m\in \Xi^-$ for some $m$. In the former case, $\rho(a_n)=c_n\varrho_n(a_n)>0$ and $\rho(a_n-)=0$, as shown in the proof of the necessity of \eqref{EQ5NIN2}. Hence 
\[
	\frac{\nu_\rho(\{a_n\})}{\rho(a_n)+\rho(a_n-)}=\frac{\rho(a_n)-\rho(a_n-)}{\rho(a_n)+\rho(a_n-)}=1=\mu(\{a_n\}). 
\]
In the latter case, $\rho(b_m-)=c_m\varrho_m(b_m-)>0$ and $\rho(b_m)=0$. Then one can also obtain $\nu_\rho(\{a_n\})/(\rho(a_n)+\rho(a_n-))=\mu(\{a_n\})$. When $\nu_\rho(\{a_n\})=0$, it is easy to find $a_n\in G^c\setminus \Xi$. As a consequence, $\mu(\{a_n\})=0$. Another endpoint $b_n$ can be treated similarly. Eventually we can conclude that \eqref{EQ3MZN} holds. In other words, the relation \eqref{EQ3MZN} between $\mu$ and $\rho$ is satisfied without extra conditions. 

(e) \emph{Conservativeness}. By Proposition~\ref{THM4}, it suffices to deal with the case that some effective interval is unbounded. More precisely, when $(L,\infty)\subset \tI_k$ (resp. $(-\infty, -L)\subset \tI_k$) for some $k$ and some constant $L>0$, it should hold
\begin{equation}\label{EQ5LXR}
	\int_L^\infty \frac{dx}{\rho(x)}\int_L^x\rho(y)dy=\infty, \quad \left(\text{resp. }\int_{-\infty}^{-L} \frac{dx}{\rho(x)}\int_x^{-L}\rho(y)dy=\infty \right). 
\end{equation}
If $(L,\infty)\cap G^c=\emptyset$ (resp. $(-\infty, -L)\cap G^c=\emptyset$), then $(L,\infty)\subset I_n$ (resp. $(-\infty, -L)\subset I_n$) for some $n$ and $\rho=c_n\varrho_n$ on $(L,\infty)$ (resp. $(-\infty, -L)$). Hence the fifth condition leads to \eqref{EQ5LXR}. When all $I_n\subset \tI_k$ is bounded, we present a sufficient condition on $\{c_n:n\geq 1\}$ for \eqref{EQ5LXR} as follows. We only treat the case $(L,\infty)\subset \tI_k$. Take a subsequence $\{I_{p_i}: i\geq 1\}$ of $\{I_n\subset \tI_k: n\geq 1\}$ such that $I_{p_i}$ increases to $\infty$ as $i\uparrow\infty$ and $L<a_{p_1}$. Then the left hand side of \eqref{EQ5LXR} is greater than
\[
	\sum_{i=1}^\infty \int_{I_{p_i}}\frac{dx}{\rho(x)}\left(\sum_{j=1}^{i-1}\int_{I_{p_j}}\rho(y)dy\right)=\sum_{i=1}^\infty \frac{B_{p_i}}{c_{p_i}}\left(c_{p_1}A_{p_1}+\cdots +c_{p_{i-1}}A_{p_{i-1}} \right).
\]
Hence if $c_{p_i}\leq B_{p_i}\left(c_{p_1}A_{p_1}+\cdots +c_{p_{i-1}}A_{p_{i-1}} \right)$ for all $i$, then \eqref{EQ5LXR} is satisfied. 

\begin{proof}[Proof of the sufficiency of Theorem~\ref{THM58}]
By these arguments, we can conclude that $\rho$ is the expected density function, if and only if $\{c_n: n\geq 1\}$ satisfies \eqref{EQ5NINL}, \eqref{EQ5NIN2} and \eqref{EQ5LXR}. It suffices to show the existence of such a set $\{c_n: n\geq 1\}$. To this end, set $\phi_n:=1/(n^2A_n)$ if $A_n<\infty$ and $\phi_n:=1$ if $A_n=\infty$, $\varphi_n:=1/(n^2V_n)$ if $V_n<\infty$ and $\varphi_n:=1$ if $V_n=\infty$. When the sequence $\{I_{p_i}: i\geq 1\}$ increasing to $\infty$ in (e) (or an analogical sequence $\{I_{q_i}\}$ decreasing to $-\infty$) exists, define $c_{p_1}:=\phi_{p_1}\wedge \varphi_{p_1}$ and by induction
\[
	c_{p_{i}}:=\phi_{p_i}\wedge \varphi_{p_i}\wedge \left(B_{p_i}\cdot\left(c_{p_1}A_{p_1}+\cdots +c_{p_{i-1}}A_{p_{i-1}} \right) \right),\quad i\geq 2.
\]
The set $\{c_{q_i}: i\geq 1\}$ is defined analogically. For $n\neq p_i,q_i$, define $c_n:=\phi_n\wedge \varphi_n$. Then one can easily check that $c_n>0$ for any $n\geq 1$, and \eqref{EQ5NINL}, \eqref{EQ5NIN2} and \eqref{EQ5LXR} hold. That completes the proof. 
\end{proof}

\subsection{Irreducibility and uniqueness}

The following corollary describes all related general skew Brownian motions and their irreducibility. The proof is trivial by Remark~\ref{RM67} and the arguments in \S\ref{SEC632}. 

\begin{corollary}\label{COR67}
Under the conditions of Theorem~\ref{THM58}, every general skew Brownian motion $X$ related to \eqref{EQ3YTT} is determined by a density function $\rho$ given as follows: 
\begin{equation}\label{EQ5RZC}
\begin{aligned}
	&\rho(z)=c_n\varrho(z),\quad z\in I_n, \quad n\geq 1 \\
	&\rho(a_n)=\lim_{z\downarrow a_n}c_n\varrho(z),\quad a_n\in \Xi^+,\\
	&\rho(z)=0,\quad z\notin G\cup \Xi^+,
\end{aligned}
\end{equation}
where $\{c_n:n\geq 1\}$ is a set of positive constants satisfying \eqref{EQ5NINL}, \eqref{EQ5NIN2} and \eqref{EQ5LXR}. Furthermore, $X$ is irreducible, if and only if $I_i$ is scale-connected to $I_j$ for all $i,j$, i.e. $G^c$ is of zero Lebesgue measure and \eqref{EQ5NEE} holds for all $i,j$.
\end{corollary}

We say the general skew Brownian motions related to \eqref{EQ3YTT} are unique, if all of them are equivalent in distribution (see \cite[\S4.2]{FOT11}). 

\begin{corollary}\label{COR614}
The general skew Brownian motions related to \eqref{EQ3YTT} are unique, if and only if for any density function $\rho$ given by \eqref{EQ5RZC}, no different intervals in $\{I_n: n\geq 1\}$ are scale-connected.
\end{corollary}
\begin{proof}
To prove the sufficiency, it suffices to note that for any density function $\rho$ given by \eqref{EQ5RZC}, the set of effective intervals must be $\{\langle a_n,b_n\rangle: n\geq 1\}$, where $a_n\in \langle a_n,b_n\rangle$ (resp. $b_n\in \langle a_n,b_n\rangle$) if and only if $\int_{a_n}^{e_n}\frac{dx}{\varrho(x)}<\infty$ (resp. $\int_{e_n}^{b_n}\frac{dx}{\varrho(x)}<\infty$). Moreover, the restriction $X|_{\langle a_n,b_n\rangle}$ of $X$ to $\langle a_n,b_n\rangle$ is an irreducible diffusion given by the density function $\varrho_n$. More precisely, it is completely characterized by the scale function $\fs_n$ and the speed measure $\fm_n$ as follows (even if $c_n\neq 1$):
\[
\fs_n(x)=\int_{e_n}^x \frac{dy}{\varrho_n(y)},\quad \fm_n(dx)=\varrho_n(x)dx.
\] 
Hence $X$ is uniquely determined. To the necessity, argue by contradiction and suppose $I_i$ is scale-connected to $I_j$ for $\rho$ in \eqref{EQ5RZC} with a set $\{c_n: n\geq 1\}$ of positive constants. Let $\tI_k$ be an effective interval of $X$ such that $I_i\cup I_j\subset \tI_k$. Take another set $\{\tilde{c}_n:n\geq 1\}$ of positive constants as follows: for $n\neq i$, set $\tilde{c}_n:=c_n$ and let $\tilde{c}_i\neq c_i$. Then $\{\tilde{c}_n:n\geq 1\}$ gives another density function $\tilde{\rho}$, which still satisfies \eqref{EQ5NINL}, \eqref{EQ5NIN2} and \eqref{EQ5LXR}. Hence the general skew Brownian motion $\tilde{X}$ with the density function $\tilde{\rho}$ is also related to \eqref{EQ3YTT}. One can easily find that $\tilde{X}$ has the same set of effective intervals as $X$. However, $\tilde{X}|_{\tI_k}$ is not equivalent to $X|_{\tI_k}$ since no constant $c>0$ exists such that $\tilde{\rho}=c\cdot \rho$ on $\tI_k$. This contradicts with the uniqueness. That completes the proof.
\end{proof}
\begin{remark}
Take $I_i$ and $I_j$ with $e_i<e_j$. Then one of the following implies that $I_i$ is not scale-connected to $I_j$:
\begin{itemize}
\item[(i)] $G^c\cap (e_i,e_j)$ is of positive Lebesgue measure;
\item[(ii)] $\int_{e_i}^{b_i}\frac{dx}{\varrho(x)}=\infty$ or $\int_{a_j}^{e_j}\frac{dx}{\varrho(x)}=\infty$;
\item[(iii)] There is an $I_n\subset (e_i,e_j)$ such that $\int_{I_n}\frac{dx}{\varrho(x)}=\infty$.
\end{itemize}
\end{remark}

From Corollaries~\ref{COR67} and \ref{COR614}, we can also conclude that if $G^c\neq \emptyset$ and there is an irreducible general skew Brownian motion related to \eqref{EQ3YTT}, then the related general skew Brownian motions are not unique. Meanwhile, there are infinite different weak solutions to \eqref{EQ3YTT} for all $x\in \bR$. 

\section{Special cases of Theorem~\ref{THM58}}\label{SEC7}

\subsection{No barriers}

In this subsection, let us consider the case without barriers:
\begin{itemize}
\item[(M1)] $G^c=\Xi=\emptyset$, i.e. $|\mu|$ is Radon on $\bR$ and $|\mu(\{z\})|<1$ for all $z\in \bR$. 
\end{itemize}
Under this assumption, write $\varrho$ for the function defined in \eqref{EQ5VIZ} with $I=\bR$ and $e=0$. 
The result below states the well-posedness of \eqref{EQ3YTT}. 

\begin{theorem}\label{THM71}
Assume (M1) and
\begin{equation}\label{EQ5XVX}
	\int_0^\infty\frac{dx}{\varrho(x)}\int_0^x\varrho(y)dy=\int_{-\infty}^0 \frac{dx}{\varrho(x)}\int_x^0\varrho(y)dy=\infty. 
\end{equation}
Then the following hold:
\begin{itemize}
\item[(1)] There exists a unique general skew Brownian motion $X$ related to \eqref{EQ3YTT}. Moreover, $X$ is irreducible.
\item[(2)] For all $x\in \bR$, the SDE \eqref{EQ3YTT} is well posed and $(X_t,\mathbf{P}_x)$ is its unique solution, where $X$ is the general skew Brownian motion in the first assertion. 
\end{itemize}
\end{theorem} 
\begin{proof}
\begin{itemize}
\item[(1)] The existence of $X$ is clear since the equivalent conditions in Theorem~\ref{THM58} are satisfied. The uniqueness and irreducibility of related general skew Brownian motions is indicated by Corollaries~\ref{COR67} and \ref{COR614}.  
\item[(2)] It suffices to show the pathwise uniqueness of \eqref{EQ3YTT} for all $x\in \bR$. Fix $x$ and let $(Y,W)$ be a weak solution to \eqref{EQ3YTT}. Set
\begin{equation}\label{EQ5FXX}
	f(z):=\left\lbrace 
	\begin{aligned}
		&\int_0^z \frac{dy}{\varrho(y)},\quad z\geq 0, \\
		&-\int_z^0 \frac{dy}{\varrho(y)},\quad z< 0. 
	\end{aligned}\right.
\end{equation}
Note that $1/\varrho$ is locally of bounded variation on $\bR$. Applying It\^o-Tanaka formula to $Y$ and $f$ (see such as \cite[pp.208~(1.5) and pp.219~(1.25)]{RY99}; though the formula there is in terms of convex functions, its proof is robust for this $f$), we have
\begin{equation}\label{EQ5FYT}
f(Y_t)-f(x)=\frac{1}{2}\int_0^t\left(\frac{1}{\varrho(Y_s-)}+\frac{1}{\varrho(Y_s)}\right)dY_s+\frac{1}{2}\int_\bR L^z_t(Y)d\left(\frac{1}{\varrho}\right)(z).
\end{equation}
Mimicking Lemma~\ref{LM58}~(5), we can deduce that 
\begin{equation}\label{EQ5VZZ}
	d\left(\frac{1}{\varrho}\right)(z)=-\left(\frac{1}{\varrho(z)}+\frac{1}{\varrho(z-)}\right)\mu(dz).
\end{equation}
Note that $dL^z_s(Y)$ is a.s. carried on $\{t\geq 0: Y_t=z\}$. 
Substituting \eqref{EQ3YTT} and \eqref{EQ5VZZ} in \eqref{EQ5FYT}, we obtain 
\[
	f(Y_t)-f(x)=\frac{1}{2}\int_0^t\left(\frac{1}{\varrho(Y_s-)}+\frac{1}{\varrho(Y_s)}\right)dW_s. 
\]
Since the discontinuous points of $1/\varrho$ are countable, it follows from the occupation times formula that the right hand side of the above equality is equal to $\int_0^t1/\varrho(Y_s)dW_s$. Let $Z_t:=f(Y_t)$ and $g:=f^{-1}$ be the inverse function of $f$. If 
\[
	f(\infty):=\int_0^\infty1/\varrho(y)dy<\infty,\text{ or }f(-\infty):=\int_{-\infty}^01/\varrho(y)dy>-\infty,\]
 set $g(z):=g(f(\infty))$ for $z\geq f(\infty)$ and $g(z):=g(f(-\infty))$ for $z\leq f(-\infty)$. Then $(Z,W)$ is a solution to
\begin{equation}\label{EQ5ZTH}
	dZ_t=h(Z_t)dW_t,\quad Z_0=f(x),
\end{equation}
where $h:=\frac{1}{\varrho}\circ g$. Clearly for any $r>0$, there exists a constant $\varepsilon_r>0$ such that $h(x)\geq \varepsilon_r$ on $[-r,r]$. Note that $1/\varrho$ is locally of bounded variation. Denote the total variation function of $1/\varrho$ by $F$, i.e. $F$ is increasing, $F(0)=0$ and $dF$ is the total bounded variation measure of $1/\varrho$. Particularly, $|1/\varrho(y)-1/\varrho(z)|\leq |F(y)-F(z)|$ for all $y,z\in \bR$. Set $\tilde{F}(z):=2\left(F\circ g\right)^2(z)$ for $z\geq 0$ and $\tilde{F}(z):=-2\left(F\circ g\right)^2(z)$ for $z< 0$. We have $\tilde{F}$ is increasing and $|h(y)-h(z)|^2\leq |\tilde{F}(y)-\tilde{F}(z)|$ for all $y,z\in \bR$. Then applying \cite[pp.360~(3.5) and pp.366~(3.13)]{RY99}, we can obtain that the pathwise uniqueness holds for \eqref{EQ5ZTH}. Since $f$ is continuous and strictly increasing, one can eventually conclude that the pathwise uniqueness holds also for \eqref{EQ3YTT}.  
\end{itemize}
That completes the proof.
\end{proof}
\begin{remark}
The condition \eqref{EQ5XVX} is used only for the conservativeness of $X$ (or in other words, the non-explosion of the solution to \eqref{EQ3YTT}). Without this condition, the well-posedness of \eqref{EQ3YTT} remains still true but a lifetime should be possibly attached, see \cite{BE13}. 
\end{remark}

The situations in \cite{LG84, ORT15, Ra11} are covered by the above theorem. In \cite{Ra11}, $\mu=\sum_{p\in \bZ} (2\alpha_p-1)\delta_{z_p}$ where $\{z_p:p\in \bZ\}$ is a set of discrete points and $\alpha_p\in (0,1)$. Clearly, (M1) is always satisfied. Instead, \cite{ORT15} considers a set of countable points with one accumulation point as presented in \eqref{EQ1LKL}. Meanwhile (M1) becomes $\sum_{k=1}^\infty |2\alpha_k^--1|+\sum_{k=-\infty}^{-1}|2\alpha_k^+-1|<\infty$. In \cite{LG84}, $\mu$ is assumed to be finite (i.e. $|\mu|$ is finite) and $|\mu(\{z\})|<1$ for all $z\in \bR$. Then from Lemma~\ref{LM58} we know that there exists a constant $c>1$ such that $1/c<\varrho(z)<c$ for all $z\in \bR$. Particularly, \eqref{EQ5XVX} is always true. Moreover, one can find that $X$ is (point) recurrent in the sense that $\mathbf{P}_x(\sigma_y<\infty)=1$ for any $x,y\in \bR$ where $\sigma_y:=\{t>0: X_t=y\}$ is the first hitting time of $X$ at $y$ (see Remark~\ref{RM5}). 

\subsection{Discrete barriers}\label{SEC53}

Next, the barrier set $\Xi$ is not imposed to be empty. Instead, we assume that $\Xi$ is discrete, i.e. $\Xi$ has no accumulation points. Then $\Xi=\{z_n: n\in S\}$, where the index set $S=\{-M, -(M-1),\cdots, 0,1,\cdots, N\}$ is a subset of $\bZ$ with $M,N\leq \infty$, can be written as a two-sided sequence:
\begin{equation}\label{EQ5ZPZ}
	\cdots < z_{-p}<z_{-(p-1)}<\cdots<z_0<z_1<\cdots <z_q<\cdots.
\end{equation}
Set $I_n:=(z_n,z_{n+1})$ and take a fixed point $e_n\in I_n$. When $N<\infty$ or $M<\infty$, $I_N:=(z_N,\infty)$ or $I_{-M-1}:=(-\infty, z_{-M})$.  Moreover, consider the following case:
\begin{itemize}
\item[(M2)] The discrete barrier set $\Xi$ is rearranged as \eqref{EQ5ZPZ} and $|\mu||_{I_n}$ is Radon on $I_n$. 
\end{itemize}
For the sake of brevity, write $n+$ and $n-$ for the intervals $(z_n,e_n]$ and $[e_{n-1},z_n)$ if no confusions cause. For example, $\mu^+(n+)$ means $\mu^+((z_n,e_n])$ and $|\mu|(n-)$ means $|\mu|([e_{n-1},z_n))$. Further write $\varrho_n$ for $\varrho_{I_n}$ in \eqref{EQ5VIZ} with $I=I_n, e=e_n$. Set a function $\varrho$ defined on $\bR\setminus \Xi$ by $\varrho|_{I_n}:=\varrho_n$. 

We classify every point $z_n\in \Xi^+$ as follows. It is called a \emph{real right barrier}, denoted by $z_n\in \Xi^+_\rr$, if
\begin{equation}\label{EQ5MNI}
	|\mu|(n+)<\infty,
\end{equation}
and
\begin{equation}\label{EQ5NYY2}
\int_{n-} \varrho(y)dy<\infty, \quad \int_{n-}\frac{dy}{\varrho(y)}=\infty. 
\end{equation}
When \eqref{EQ5MNI} holds and $\varrho|_{n-}$ can be extended to a function of bounded variation on $[e_{n-1},z_n]$ with
\begin{equation}\label{EQ5VZN}
	\varrho(z_n-):=\lim_{z\uparrow z_n}\varrho(z)=0,\quad \int_{n-}\frac{dy}{\varrho(y)}<\infty,
\end{equation}
we call $z_n$ a \emph{pseudo right barrier} and write $z_n\in \Xi^+_\rp$. 
Clearly $\Xi^+_\rr\cap \Xi^+_\rp=\emptyset$. Set further $\Xi^+_\rn:=\Xi^+\setminus (\Xi^+_\rr\cup\Xi^+_\rp)$. Then a point in $\Xi^+_\rn$ is called a \emph{nonsensical right barrier}. Similarly, one can define the sets $\Xi^-_\rr$, $\Xi^-_\rp$ and $\Xi^-_\rn$ of real, pseudo, and nonsensical left barriers. Let $\Xi_\rr:=\Xi^+_\rr\cup \Xi^-_\rr$, $\Xi_\rp:=\Xi^+_\rp\cup \Xi^-_\rp$ and $\Xi_\rn:=\Xi^+_\rn\cup \Xi^-_\rn$.

\begin{theorem}\label{THM7}
Assume (M2) and that whenever a constant $L>0$ exists such that $(L,\infty)\cap \Xi=\emptyset$ (resp. $(-\infty, -L)\cap \Xi=\emptyset$),
\begin{equation}\label{EQ5LXV}
	\int_L^\infty\frac{dx}{\varrho(x)}\int_L^x \varrho(y)dy=\infty,\quad \left(\text{resp. }\int_{-\infty}^{-L} \frac{dx}{\varrho(x)}\int_x^{-L} \varrho(y)dy=\infty \right).
\end{equation}
Then the following hold:
\begin{itemize}
\item[(1)] There exists a general skew Brownian motion related to \eqref{EQ3YTT}, if and only if $\Xi_\rn$ is empty. Meanwhile, all general skew Brownian motions related to \eqref{EQ3YTT} share the same set of effective intervals.
\item[(2)] One (or all) general skew Brownian motion related to \eqref{EQ3YTT} is irreducible, if and only if $\Xi_\rr$ is empty. 
\item[(3)] The general skew Brownian motions related to \eqref{EQ3YTT} are unique, if and only if $\Xi_\rp$ is empty.
\end{itemize} 
\end{theorem}
\begin{proof}
To show the first assertion, suppose $\Xi_\rn=\emptyset$ at first. It suffices to verify the conditions in Theorem~\ref{THM58} and then we can conclude the existence of related general skew Brownian motions. Indeed, it follows from $\Xi_\rn=\emptyset$ that $|\mu|(n+)<\infty$ (resp. $|\mu|(n-)<\infty$) for all $z_n\in \Xi^+$ (resp. $z_n\in \Xi^-$). Thus the first condition in Theorem~\ref{THM58} is satisfied. The second condition in Theorem~\ref{THM58} is clear since $G^c=\Xi$ under the assumption (M2). When $z_n\in \Xi^+$, \eqref{EQ5MNI} indicates that $\varrho$ is bounded on $n+$ by Lemma~\ref{LM58}~(7) and thus $\int_{n+}\varrho(y)dy<\infty$. When $z_n\in \Xi^-=\Xi^-_\rr\cup \Xi^-_\rp$, the definition of real or pseudo left barrier also leads to $\int_{n+}\varrho(y)dy<\infty$. Hence the third condition in Theorem~\ref{THM58} is true. The fourth condition in Theorem~\ref{THM58} is implied by \eqref{EQ5VZN}. Finally, \eqref{EQ5LXV} is nothing but \eqref{EQ5LXV2}. Next, suppose $\Xi_n\neq \emptyset$ and we will see no related general skew Brownian motions exist. Without loss of generality, let $z_n\in \Xi^+_\rn$. Since $z_n\notin \Xi^+_\rr$, we have $|\mu|(n+)=\infty$ or $\int_{n-}\varrho(y)dy=\infty$ or $\int_{n-}\frac{dy}{\varrho(y)}<\infty$. The former two cases do not admit the existence of related general skew Brownian motions by Theorem~\ref{THM58}. Hence $\int_{n-}\frac{dy}{\varrho(y)}<\infty$. But $z_n\notin \Xi^+_\rp$ implies that either $\varrho|_{n-}$ cannot be extended to a function of bounded variation on $[e_{n-1},z_n]$, or \eqref{EQ5VZN} is not true. In the former case, the fourth condition in Theorem~\ref{THM58} is not satisfied. In the latter case, we must have $\varrho(z_n-)>0$. Note that $\varrho(z_n)>0$ by $|\mu|(n+)<\infty$. Then $|\mu|(\{z_n\})=\left|\frac{\varrho(z_n)-\varrho(z_n-)}{\varrho(z_n)+\varrho(z_n-)}\right|<1$, which contradicts with $z_n\in \Xi^+$. Therefore the existence of related general skew Brownian motions is equivalent to $\Xi_\rn=\emptyset$. Note that when $z_n\in \Xi_p$ (resp. $z_n\in \Xi_\rr$), 
\[
	\int_{e_{n-1}}^{e_n}\frac{dy}{\varrho(y)}<\infty,\quad \left(\text{resp. } \int_{e_{n-1}}^{e_n}\frac{dy}{\varrho(y)}=\infty\right). 
\]
Since every general skew Brownian motion $X$ related to \eqref{EQ3YTT} has the density function $\rho$ given by \eqref{EQ5RZC}, it follows that each effective interval of $X$ is ended by real barriers or $\pm \infty$, and all barriers belonging to the interior of some effective interval are pseudo. Particularly, the set of effective intervals is independent of the choice of $\{c_n: n\geq 1\}$ in \eqref{EQ5RZC}. In other words, all general skew Brownian motions related to \eqref{EQ3YTT} share the same set of effective intervals. 
Since the irreducibility of $X$ means that $X$ has exactly one effective interval, one can conclude that it is also equivalent to $\Xi_\rr=\emptyset$ by the above arguments. For the third assertion, note that Corollary~\ref{COR67} tells us the uniqueness holds if and only if for any $\rho$ given by \eqref{EQ5RZC}, no different intervals in $\{I_n: -M-1\leq n\leq N\}$ are scale-connected. Hence it is equivalent to $\Xi_\rp=\emptyset$. 
That  completes the proof. 
\end{proof}

The name of nonsensical barriers comes from the consequence that if $\Xi_\rn\neq \emptyset$, then there are no general skew Brownian motions related to \eqref{EQ3YTT}. Pseudo barriers are not ``real"  because related general skew Brownian motions can go across them from both sides. Only real barriers are definitely effective in the sense that no related general skew Brownian motions can pass through them from either side. 

As stated in Corollary~\ref{COR67}, the density function $\rho$ of a related general skew Brownian motion $X$ is determined by a set of positive constant $\{c_n: n\geq 1\}$ in the manner of \eqref{EQ5RZC}. Recall that $\{c_n:n\geq 1\}$ should satisfy \eqref{EQ5NINL}, \eqref{EQ5NIN2} and \eqref{EQ5LXR}. Under (M2), 
\eqref{EQ5NINL} and \eqref{EQ5NIN2} are always true. Hence only the conservativeness of $X$ is required for the choice of $\{c_n: n\geq 1\}$. On the other hand, the effective intervals $\{\tI_k: k\geq 1\}$ of $X$ are independent of $\{c_n: n\geq 1\}$. Each $\tI_k$ is ended by real barriers or $\pm \infty$.  When $\Xi_\rp\neq \emptyset$ and $\tI_k$ contains pseudo barriers, there are infinite related general skew Brownian motions, whose restriction to $\tI_k$ are different. Particularly for $x\in \tI_k$, the pathwise uniqueness of \eqref{EQ3YTT} does not hold. 

\begin{remark}\label{RM74}
Consider the simple case $\mu=\delta_0$. One can easily find $\Xi=\Xi^+_\rn=\{0\}$ and hence Theorem~\ref{THM7} tells us there are no general skew Brownian motions related to \eqref{EQ3YTT}. However, it is well known that \eqref{EQ3YTT} is well posed and the unique solution is the reflected Brownian motion on $[0,\infty)$ as shown in \cite{BC05}, \cite{HS81} and \cite{LG84}.
This is not surprising, because the general skew Brownian motion we expect is a diffusion process on $\bR$, but the real state space of reflected Brownian motion is $[0,\infty)$ and $(-\infty,0)$ is treated as its ceremony. More precisely, the solution to \eqref{EQ3YTT} in the above citations is understood as a diffusion process $(X_t,\mathbf{P}_x)$ with the lifetime $\zeta=\inf\{t>0:X_t\notin [0,\infty)\}$ such that $\mathbf{P}_x(\zeta=\infty)=1$ for $x\geq 0$.

Generally  we can also derive the well-posedness of \eqref{EQ3YTT} by attaching a suitable lifetime to the solutions. Indeed, consider $I_n=(z_n,z_{n+1})$ and assume that
\begin{itemize}
\item[(i)] $z_n\in \Xi^+$ and $|\mu|(n+)<\infty$ whenever $\int_{n+}\frac{dy}{\varrho(y)}<\infty$;
\item[(ii)] $z_{n+1}\in \Xi^-$ and $|\mu|((n+1)-)<\infty$ whenever $\int_{(n+1)-}\frac{dy}{\varrho(y)}<\infty$. 
\end{itemize}  
Let $\tI_n:=\langle z_n,z_{n+1}\rangle$, where $z_n\in \tI_n$ (resp. $z_{n+1}\in \tI_n$) if and only if $\int_{n+}\frac{dy}{\varrho(y)}<\infty$ (resp. $\int_{(n+1)-}\frac{dy}{\varrho(y)}<\infty$). Then mimicking the proof of Theorem~\ref{THM71}, one can conclude that for any $x\in \tI_n$ there exists a unique solution $(X_t,\mathbf{P}_x)$ to \eqref{EQ3YTT} such that 
\begin{equation}\label{EQ5PXN}
	\mathbf{P}_x(\zeta_n=\infty)=1,
\end{equation}
where $\zeta_n:=\inf\{t>0:X_t\notin \tI_n\}$. 
Particularly when $\Xi_\rn=\Xi_\rp=\emptyset$, these assumptions (i) and (ii) are always true for all $n$ and $\{\tI_n: n\geq 1\}$ is exactly the set of effective intervals obtained in Theorem~\ref{THM71}. The unique solution $(X_t,\mathbf{P}_x)$ to \eqref{EQ3YTT} satisfying \eqref{EQ5PXN} coincides with the restriction of the unique general skew Brownian motion related to \eqref{EQ3YTT} to $\tI_n$. 

In the case $\Xi_\rn=\emptyset$ but $\Xi_\rp\neq \emptyset$, take an effective interval $\tI_k$ containing pseudo barriers. Then there are different weak solutions to \eqref{EQ3YTT} for $x\in \tI_k$ as mentioned before this remark. However, it is still possible that some $I_n\subset \tI_k$ satisfies (i) and (ii). An example is given in Example~\ref{EXA13}. When $0<\alpha<1$ in this example, $0$ is a pseudo barrier but $I_1=(0,\infty)$ satisfies (i) and (ii). As a consequence, if attaching the lifetime $\zeta_1:=\inf\{t>0: X_t\notin [0,\infty)\}$, the weak solutions to \eqref{EQ3YTT} (with $x\geq 0$) satisfying $\mathbf{P}_x(\zeta_1=\infty)=1$ are unique and identified with reflected Brownian motion. 
\end{remark}

A lemma below provides comprehensible conditions to distinguish different barriers. Clearly, $z_n\in \Xi^+$ with $|\mu|(n+)=\infty$ is always a nonsensical barrier.

\begin{lemma}\label{LM10}
Assume (M2) and take $z_n\in \Xi^+$ with $|\mu|(n+)<\infty$. Then the following hold:
\begin{itemize}
\item[(1)] If $\mu^-(n-)<\infty$, then $z_n\in \Xi^+_\rn$. 
\item[(2)] If $\mu^-(n-)=\infty$ and $\mu^+(n-)<\infty$, then $z_n\in \Xi^+_\rr\cup \Xi^+_\rp$. Meanwhile, $z_n\in \Xi^+_\rr$ (resp. $z_n\in \Xi^+_\rp$) if and only if
\[
	\int_{n-}\frac{dy}{\varrho(y)}=\infty,\quad \left(\text{resp. }\int_{n-}\frac{dy}{\varrho(y)}<\infty\right).
\]
\item[(3)] If $\mu^-(n-)=\mu^+(n-)=\infty$, then $z_n$ is possibly a real, pseudo or nonsensical barrier. 
\end{itemize}
\end{lemma}
\begin{proof}
\begin{itemize}
\item[(1)] Suppose $\mu^-(n-)<\infty$. If $\mu^+(n-)<\infty$, then it follows from Lemma~\ref{LM58} that $\varrho(z_n-)$ exists and is positive. Then neither \eqref{EQ5NYY2} nor \eqref{EQ5VZN} is satisfied. Thus $z_n\in \Xi^+_\rn$. If $\mu^+(n-)=\infty$, one can  deduce that $\lim_{z\uparrow z_n}\frac{1}{\varrho}(z)=0$ and $1/\varrho$ is of bounded variation on $[e_{n-1},z_n]$ by mimicking Lemma~\ref{LM58}~(6). Particularly, $1/\varrho$ is integral on $[e_{n-1},z_n]$ but $\varrho(z_n-)\neq 0$. We still have $z_n\in \Xi^+_\rn$. 
\item[(2)] When $\mu^-(n-)=\infty$ and $\mu^+(n-)<\infty$, Lemma~\ref{LM58}~(6) tells us $\varrho(z_n-)=0$ and $\varrho$ is of bounded variation on $[e_{n-1},z_n]$. Particularly, $\varrho$ is integral on $[e_{n-1},z_n]$. If $\int_{n-}\frac{dy}{\varrho(y)}=\infty$, we have \eqref{EQ5NYY2} holds and $z_n\in \Xi^+_\rr$. Otherwise \eqref{EQ5VZN} holds and $z_n\in \Xi^+_\rp$. 
\item[(3)] We shall raise examples to show their possibilities in Example~\ref{EXA14}.
\end{itemize}
That completes the proof.
\end{proof}

Then we can obtain a useful corollary by ignoring the situation $\mu^-(n-)=\mu^+(n-)=\infty$. The proof is straightforward by applying Theorem~\ref{THM7} and Lemma~\ref{LM10}. 

\begin{corollary}\label{COR7}
Under the same conditions as Theorem~\ref{THM7} assume further that no $n$ exists such that $\mu^+(n-)=\mu^-(n-)=\infty$ or $\mu^+(n+)=\mu^-(n+)=\infty$. Then there exists a general skew Brownian motion related to \eqref{EQ3YTT}, if and only if 
\begin{itemize}
\item[(1)] For any $z_n\in \Xi^+$, $|\mu|(n+)<\infty$, $\mu^+(n-)<\infty$ and $\mu^-(n-)=\infty$; 
\item[(2)] For any $z_n\in \Xi^-$, $|\mu|(n-)<\infty$, $\mu^+(n+)=\infty$ and $\mu^-(n+)<\infty$.
\end{itemize}
Meanwhile, $z_n\in \Xi^\pm$ is a real (resp. pseudo) barrier, if and only if $\int_{n\pm}1/\varrho(y)dy=\infty$ (resp. $\int_{n\pm}1/\varrho(y)dy<\infty$). 
\end{corollary}

We complete this subsection with two examples.  

\begin{example}\label{EXA13}
Take a constant $\alpha$ and
\[
	\mu(dz)=-\frac{\alpha}{2} |z|^{-1}dz|_{(-\infty,0)} +\delta_0. 
\]
Write $I_0=(-\infty, 0), I_1=(0,\infty)$, and take the fixed point in $I_0$ and $I_1$ to be $e_0=-1$ and $e_1=1$. Clearly, $\Xi=\Xi^+=\{z_1: z_1=0\}$, $|\mu|(1+)=0$ and (M2) is true. It is easy to obtain $\varrho(z)=|z|^\alpha$ on $I_1$ and $\varrho\equiv 1$ on $I_2$ (cf. Example~\ref{EXA3}) and straightforward to verify \eqref{EQ5LXV}. Hence all assumptions in Corollary~\ref{COR7} are satisfied. Note that $\mu^+(1-)=0, \mu^-(1-)=\infty$ if and only if $\alpha>0$.  Then from Corollary~\ref{COR7} one can conclude that general skew Brownian motions related to \eqref{EQ3YTT} exist if and only if $\alpha>0$. Moreover, $0$ is a real (resp. pseudo, nonsensical) barrier if and only if $\alpha\geq 1$ (resp. $0<\alpha<1$, $\alpha\leq 0$) by their definitions. 

When $0<\alpha<1$, every general skew Brownian motion related to \eqref{EQ3YTT} is given by the density function 
\[
	\rho(z)=|z|^\alpha\cdot 1_{(-\infty,0)}(z)+c_1\cdot 1_{[0,\infty)}(z)
\]
for some constant $c_1>0$. All of them are irreducible and (point) recurrent. Accordingly, the pathwise uniqueness of \eqref{EQ3YTT} does not hold for any $x\in \bR$. 
\end{example}

Another example illustrates the three cases possibly appeared in the last assertion of Lemma~\ref{LM10}. 

\begin{example}\label{EXA14}
Let
\[
	\mu(dz)=\sum_{k=1}^\infty \left(\beta^+_k \delta_{\{-\frac{1}{2k}\}}-\beta^-_k\delta_{\{-\frac{1}{2k+1}\}}\right) +\delta_0,
\]
where $\beta^+_k, \beta^-_k$ are positive constants with $0<\beta^+_k,\beta^-_k<1$ and
\begin{equation}\label{EQ5KBK}
	\sum_{k\geq 1}\beta^+_k=\sum_{k\geq 1}\beta^-_k=\infty.
\end{equation}
Let $I_0, I_1$ and $e_0, e_1$ be the same as those in Example~\ref{EXA13}. Then $z_1=0\in \Xi^+$, $|\mu|(1+)=0$ and $\mu^+(1-)=\mu^-(1-)=\infty$. Note that (M2) and \eqref{EQ5LXV} are obviously satisfied. 
\begin{itemize}
\item[(1)] Take $\beta_k^+=\beta^-_k=\frac{k}{k+2}$. One can deduce that $\varrho=k+1$ on $(-\frac{1}{2k},-\frac{1}{2k+1})$ and $\varrho=1$ on $(-\frac{1}{2k+1},-\frac{1}{2k+2})$. Then $0$ is neither a real barrier nor a pseudo barrier. We have $0\in \Xi^+_\rn$.
\item[(2)] Take a constant $\alpha>0$ and 
\[
	\beta^+_k=\frac{\left(\frac{k+1}{k}\right)^\alpha-1}{\left(\frac{k+1}{k}\right)^\alpha+1},\quad \beta^-_k=\frac{1-\left(\frac{k}{k+1}\right)^{2\alpha}}{1+\left(\frac{k}{k+1} \right)^{2\alpha}}.
\]
One can check that \eqref{EQ5KBK} is satisfied and $\varrho\equiv \left(\frac{k+1}{k^2} \right)^\alpha$ on $(-\frac{1}{2k}, -\frac{1}{2k+1})$ and $\varrho\equiv \left(\frac{1}{k+1}\right)^\alpha$ on $(-\frac{1}{2k+1},-\frac{1}{2k+2})$. Clearly $\int_{-1}^0\varrho(y)dy<\infty$ since $\varrho$ is bounded. Moreover,
\[
	\int_{-1}^0\frac{dy}{\varrho(y)}=\sum_{k\geq 1}\left(\frac{k^2}{k+1} \right)^\alpha\cdot \frac{1}{2k(2k+1)}+\sum_{k\geq 1}\frac{1}{(k+1)^\alpha}\cdot \frac{1}{(2k+1)(2k+2)}.
\]
When $\alpha\geq 1$, we have $\int_{-1}^0\frac{dy}{\varrho(y)}=\infty$ and $0$ is a real barrier. When $0<\alpha<1$, $\int_{-1}^0\frac{dy}{\varrho(y)}<\infty$ and $\varrho(0-)=0$. One can also conclude that $\varrho$ is of bounded variation on $[-1,0]$ since
\[
\sum_{k\geq 1}\left|\left(\frac{k+1}{k^2}\right)^\alpha- \left(\frac{1}{k+1}\right)^\alpha\right|+\sum_{k\geq 1}\left|\left(\frac{k+2}{(k+1)^2}\right)^\alpha- \left(\frac{1}{k+1}\right)^\alpha\right|<\infty. 
\]
Therefore $0$ is a pseudo barrier.
\end{itemize}
\end{example}

\subsection{Barriers with many accumulation points}\label{SEC54}

Let $K$ be a generalized Cantor set. More precisely, take a sequence $\{\alpha_j:j\geq 1\}$ of numbers in $(0,1)$ and define a decreasing sequence $\{K_j: j\geq 1\}$ of closed sets as follows: $K_0:=[0,1]$, $K_j$ is obtained by removing the open middle $\alpha_j$th from each of the intervals that make up $K_j$. Then $K:=\cap_{j\geq 1}K_j$ is called a generalized Cantor set (see \cite{F99}). When $\alpha_j\equiv 1/3$, $K$ is nothing but the standard Cantor set. Write $K^c$ as a union of disjoint open intervals:
\begin{equation}\label{EQ5KNA}
	K^c=\cup_{n\geq 1}(a_n, b_n)
\end{equation}
and set $I_n:=(a_n, b_n)$. For the sake of clearness, assume $I_1=(-\infty, 0)$ and $I_2=(1,\infty)$. Take $e_1:=-1, e_2:=2$ and $e_n:=(a_n+b_n)/2$ for $n\geq 3$.

We turn to consider a more complicated case of Theorem~\ref{THM58} that $\mu$ satisfies
\begin{itemize}
\item[(M3)] $\Xi^+=\{a_n:n\geq 2\}$, $\Xi^-=\{b_n: n=1\text{ or }n\geq 3\}$ and $|\mu|$ is Radon on $I_n$ for all $n$.  
\end{itemize}
Note that (M3) indicates $G=K^c$ where $G$ is in Theorem~\ref{THM58}. Define a function $\varrho_n$ on $I_n$ by \eqref{EQ5VIZ} with $I=I_n$ and set $\varrho:=\varrho_n$ on all $I_n$ as before. When $|\mu|(I_n)<\infty$ for $n\geq 3$, $\varrho_n$ can be extended to a function of bounded variation on $[a_n,b_n]$ due to Lemma~\ref{LM58}. Denote its canonical version still  by $\varrho_n$ and the induced Radon signed measure on $[a_n,b_n]$ by $\nu_{\varrho_n}$. 


\begin{theorem}\label{THM8}
Assume (M3) and that 
\begin{equation}\label{EQ5XXX}
	\int_2^\infty\frac{dx}{\varrho(x)}\int_2^x \varrho(y)dy=\int_{-\infty}^{-1} \frac{dx}{\varrho(x)}\int_x^{-1} \varrho(y)dy=\infty.
\end{equation}
Then there exists a general skew Brownian motion related to \eqref{EQ3YTT} if and only if $|\mu|([-1,0]\cup [1,2])<\infty$, $|\mu|(I_n)<\infty$ for $n\geq 3$ and $|\mu|(K\setminus \Xi)=0$. Furthermore, every general skew Brownian motion $X$ related to \eqref{EQ3YTT} is determined by a density function $\rho$ given by \eqref{EQ5RZC} with a set of positive constants $\{c_n: n\geq 1\}$ such that
\begin{itemize}
\item[(1)] $\sum_{n\geq 3}c_nA_n<\infty$ where $A_n:=\int_{I_n}\varrho(y)dy$;
\item[(2)] for all $e_i<e_j$ such that $K\cap (e_i,e_j)$ is of zero Lebesgue measure and
\[
\sum_{n: I_n\subset (e_i,e_j)}\frac{B_n}{c_n}<\infty,\quad \left(B_n:=\int_{I_n}\frac{dy}{\varrho(y)}\right),
\]
it holds 
\[
\sum_{n: I_n\subset (e_i,e_j)}c_nV_n<\infty,
\]
where $V_n=|\nu_{\varrho_n}|([a_n,b_n])$.
\end{itemize}
Particularly, $X$ is irreducible, if and only if $K$ is of zero Lebesgue measure and
\begin{equation}\label{EQ5NCN}
\sum_{n\geq 3}\left(c_nA_n+\frac{B_n}{c_n}+c_nV_n \right)<\infty. 
\end{equation}
\end{theorem}
\begin{proof}
Note that the first condition in Theorem~\ref{THM58} is now read as $|\mu|([-1,0]\cup [1,2])<\infty$ and $|\mu|(I_n)<\infty$ for $n\geq 3$. Meanwhile, the third and fourth conditions there are always satisfied by this condition. The conservativeness is  indicated by \eqref{EQ5XXX}. Hence we can conclude the assertion concerning the existence. Moreover, the expression of $X$ is clear by applying Corollary~\ref{COR67}. Finally, the irreducibility of $X$ is equivalent to that all $I_n$ are scale-connected. By observing the structure of $\{I_n: n\geq 1\}$, one can find that this is also equivalent to the scale-connection between $I_1$ and $I_2$. In other words, $K$ is of zero Lebesgue measure and $\sum_{n\geq 3}B_n/c_n<\infty$. Note that $\sum_{n\geq 3}(c_nA_n+c_nV_n)<\infty$ follows from the requirements of $\{c_n: n\geq 1\}$. That completes the proof. 
\end{proof}

The corollary below is straightforward by Theorem~\ref{THM8}. Note that for $\mu$ in \eqref{EQ5MND}, $\varrho\equiv 1$ on $K^c$ and as a consequence, $A_n=B_n=|b_n-a_n|$ and $V_n=2$ for $n\geq 3$. 

\begin{corollary}\label{COR8}
Let $K$ be the generalized Cantor set as above and assume that 
\begin{equation}\label{EQ5MND}
	\mu=\sum_{n\neq 1}\delta_{a_n}-\sum_{n\neq 2}\delta_{b_n}.
\end{equation}
Then there always exists a general skew Brownian motion related to \eqref{EQ3YTT}. Every related general skew Brownian motion $X$ is determined by a density function 
$\rho$ given by \eqref{EQ5RZC} with a set of positive constants $\{c_n: n\geq 1\}$ such that
\begin{itemize}
\item[(1)] $\sum_{n\geq 3} c_n\cdot |b_n-a_n|<\infty$; and
\item[(2)] For $e_i<e_j$ such that $\sum_{n: I_n\subset (e_i,e_j)}\frac{|b_n-a_n|}{c_n}<\infty$ and $K\cap (e_i,e_j)$ is of zero Lebesgue measure, it holds that $\sum_{n: I_n\subset (e_i,e_j)}c_n<\infty$. 
\end{itemize}
Furthermore, $X$ is irreducible, if and only if $K$ is of zero Lebesgue measure and 
\begin{equation}\label{EQ5NBN}
	\sum_{n\geq 3}\left(\frac{|b_n-a_n|}{c_n}+c_n \right)<\infty. 
\end{equation}
\end{corollary}

We show some facts about the uniqueness of related general skew Brownian motions. 
Recall that $K$ is produced by a sequence $\{\alpha_j\}$ of numbers in $(0,1)$. When $\sum_{j\geq 1}\alpha_j<\infty$, we know from \cite[\S1.5]{F99} that $K$ is of positive Lebesgue measure, and $K\cap (e_i,e_j)$ is also of positive Lebesgue measure for any $i,j$. This indicates that  every two intervals in $\{I_n: n\geq 1\}$ cannot be scale-connected. Particularly,  there is a unique general skew Brownian motion $X$ related to \eqref{EQ3YTT} in Theorem~\ref{THM8} and the set of effective intervals of $X$ must be $\{[a_n,b_n]: n\geq 1\}$ (set $[a_1,b_1]:=(-\infty, 0]$, $[a_2,b_2]:=[1,\infty)$). For the case \eqref{EQ5MND},  the restriction of $X$ to $[a_n,b_n]$ is always a reflected Brownian motion. Another corollary summarizes the typical case $\alpha_j\equiv \alpha$, where $K$ is always of zero Lebesgue measure.

\begin{corollary}\label{COR510}
Let $K$ be a generalized Cantor set produced by $\alpha_j\equiv \alpha\in (0,1)$ and assume that \eqref{EQ5MND} holds. 
\begin{itemize}
\item[(1)] When $\alpha\geq 1/4$, there exists a unique general skew Brownian motion related to \eqref{EQ3YTT}, whose restriction to $[a_n,b_n]$ is a reflected Brownian motion. 
\item[(2)] When $\alpha<1/4$, there are infinite irreducible general skew Brownian motions related to \eqref{EQ3YTT}. Particularly, for all $x\in \bR$, \eqref{EQ3YTT} has infinite weak solutions. 
\end{itemize}
\end{corollary}
\begin{proof}
Assume that $|b_n-a_n|$ is decreasing in $n$. In other words, $|b_3-a_3|=\alpha$ and for $p\geq 1$ and 
\begin{equation}\label{EQ5PNP}
	2+2^0+\cdots 2^{p-1}<n\leq 2+2^0+\cdots +2^{p},
\end{equation}
we have $|b_n-a_n|=\alpha^{p+1}$.

Firstly, consider the case $\alpha\geq 1/4$.  It suffices to show every two different intervals in $\{I_n: n\geq 1\}$ are not scale-connected. Without loss of generality we only prove that $I_1$ and $I_2$ are not scale-connected, i.e. \eqref{EQ5NBN} cannot be true. Argue by contradiction. Indeed, \eqref{EQ5NBN} implies $\sum_{n\geq 3} |b_n-a_n|^{1/2}<\infty$. However it follows from $\alpha\geq 1/4$ that
\[
	\sum_{n\geq 3} |b_n-a_n|^{1/2}=\sum_{p\geq 1}2^p\cdot \left(\alpha^{p+1} \right)^{1/2}=\alpha^{1/2}\sum_{p\geq 1} \left( 2\alpha^{1/2}\right)^p=\infty, 
\]
which leads to contradiction. 
Hence $I_1$ is not scale-connected to $I_2$.

Next, consider the case $\alpha<1/4$. Take $\beta$ such that $2\alpha<\beta<1/2$. Set $c_1=c_2=c_3:=1$ and for $n$ in \eqref{EQ5PNP}, set
\[
	c_n:= \beta^p. 
\]
Define $\rho$ as in Corollary~\ref{COR8}. 
Clearly, $\sum_{n\geq 3}c_n\cdot |b_n-a_n|<\sum_{n\geq 3}c_n\leq \sum_{p\geq 0}\beta^p\cdot 2^p<\infty$ since $2\beta<1$. 
Moreover, 
\[
\sum_{n\geq 3}\frac{|b_n-a_n|}{c_n}=\sum_{p\geq 0}\alpha^{p+1}\cdot \beta^{-p} \cdot 2^p=\alpha\sum_{p\geq 0}\left(\frac{2\alpha}{\beta}\right)^p<\infty. 
\]
Hence $\rho$ leads to an irreducible general skew Brownian motion related to \eqref{EQ3YTT}.  By taking different $\beta$ in $(2\alpha, 1/2)$, one can obtain different general skew Brownian motions related to \eqref{EQ3YTT}. That completes the proof. 
\end{proof}

\appendix
\section{Basics of symmetric one-dimensional diffusions}\label{SEC2}

In this appendix, we are concerned with the basics of a symmetric one-dimensional diffusion $X=(X_t,\bfP_x)$ associated with a regular and strongly local Dirichlet form $(\EE,\FF)$ on $L^2(\bR,\fm)$ where $\fm$ is a given fully supported positive Radon measure on $\bR$.  The crucial fact is that $(\EE,\FF)$ as well as $X$ can be represented by a set of so-called effective intervals, as shown in \cite{LY172}. We shall repeat this result below for readers' convenience. Based on this representation, the quasi-notions and the conservativeness of $(\EE,\FF)$ will be further characterized. Note that all these characterizations can be easily extended to the local case by applying the killing and resurrected transforms, as shown in \cite[\S4]{LY172}. 

\subsection{Representation of symmetric one-dimensional diffusions}\label{SEC21}

Given an interval $\tI=\langle \fa, \fb\rangle$, $\fm|_\tI$ stands for the restriction of $\fm$ to $\tI$. A scale function $\fs$ on $\tI$ means a continuous and strictly increasing function on it. For the sake of brevity, we always take a fixed point $\fe$ in the interior of $\tI$ and impose $\fs(\fe)=0$. The induced Radon measure of $\fs$ is denoted by $\lambda_\fs$. Denote the family of all scale functions on $\tI$ by $\mathrm{S}(\tI)$, i.e. 
\[
	\mathrm{S}(\tI)=\{\fs:\tI\rightarrow \bR \text{ continuous, strictly increasing and }\fs(\fe)=0\}. 
\] 
With a scale function $\fs\in \mathrm{S}(\tI)$, one can construct a regular and strongly local Dirichlet form on $L^2(\tI,\fm|_\tI)$:
\begin{equation}\label{EQ2FSFL}
\begin{aligned}
	&\FF^{(\mathfrak{s})}:=\bigg\{f\in L^2(\mathtt{I}, \fm|_{\tI}): f\ll \fs, \frac{du}{d\fs}\in L^2(\tI,d\fs); \\
	& \qquad \qquad \qquad 	 f(\fa)=0\; (\text{resp. } f(\fb)=0) \text{ whenever }(\text{L}) \; (\text{resp. (R)})\bigg\}, \\
	&\EE^{(\fs)}(f,g):=\frac{1}{2}\int_{\tI}\frac{df}{d\fs}\frac{df}{d\fs}d\fs,\quad f,g\in \FF^{(\fs)},
\end{aligned}\end{equation}
which is associated with an $\fm|_{\tI_k}$-symmetic irreducible diffusion $X^{(\fs)}$ on $\tI$ with the scale function $\fs$ (see \cite{FHY10, LY172}).
Here, (L) and (R) stand for the conditions at the left and right endpoints of $\tI$:
\begin{itemize}
\item[(L)] $\fa=-\infty$ and $\lambda_\fs((-\infty, \mathfrak{e}))+\fm((-\infty, \mathfrak{e}))<\infty$;
\item[(R)]  $\fb=\infty$ and $\lambda_\fs((\mathfrak{e},\infty))+\fm((\mathfrak{e},\infty))<\infty$.
\end{itemize}
Moreover, we say $\fs\in \mathrm{S}(\tI)$ is adapted (to $\tI)$ if $\fa\in \tI$ (resp. $\fb\in \tI$) is equivalent to $\fa+\fs(\fa)>-\infty$ (resp. $\fb+\fs(\fb)<\infty$). The family of all adapted scale functions on $\tI$ is denoted by $\mathrm{S}_\infty(\tI)$. 

The following theorem is taken from \cite[\S2.4]{LY172}. In fact, it gives  the irreducible decomposition of a symmetric diffusion on $\bR$. 

\begin{theorem}\label{THM0}
Let $\fm$ be a fully supported positive Radon measure on $\mathbb{R}$. Then $(\EE, \FF)$ is a regular and strongly local Dirichlet form on $L^2(\mathbb{R},m)$ if and only if there exists a set of at most countable disjoint intervals $\{\tI_k=\langle \fa_k,\fb_k\rangle: \tI_k\subset \mathbb{R}, k\geq 1\}$ and a scale function $\fs_k\in \mathrm{S}_\infty(\tI_k)$ for each $k\geq 1$ such that
\begin{equation}\label{EQ2FULR}
\begin{aligned}
	&\FF=\left\{f\in L^2(\mathbb{R}, \fm): f|_{\tI_k}\in \FF^{(\fs_k)}, \sum_{k\geq 1}\EE^{(\fs_k)}(f|_{\tI_k}, f|_{\tI_k})<\infty  \right\},  \\
	&\EE(f,g)=\sum_{k\geq 1}\EE^{(\fs_k)}(f|_{\tI_k}, g|_{\tI_k}),\quad f,g \in \FF,
\end{aligned}
\end{equation}
where for each $k\geq 1$, $(\EE^{(\fs_k)}, \FF^{(\fs_k)})$ is given by \eqref{EQ2FSFL} with $(\tI_k,\fs_k)$ in place of $(\tI, \fs)$.  Moreover, the intervals $\{\tI_k: k\geq 1\}$ and the scale functions $\{\fs_k:k\geq 1\}$ are uniquely determined, if the difference of order is ignored.
\end{theorem}

Following \cite{LY172}, we call $\tI_k$ or $(\tI_k, \fs_k)$ an effective interval of $(\EE,\FF)$ or $X$ if no confusions cause. These effective intervals with the symmetric measure $\fm$ determine $(\EE,\FF)$ and $X$ completely. When restricting to $\tI_k$, $X$ is a ``regular" diffusion process with the scale function $\fs_k$, the speed measure $\fm|_{\tI_k}$ and no killing inside (see \cite[Chapter V~\S6]{RW87}). Particularly, $X$ is irreducible if and only if exactly one effective interval $\tI_1=\bR$ appears. Moreover, every point outside $\cup_{k\geq 1}\tI_k$ is a trap of $X$ in the sense that 
\[
	\mathbf{P}_x(X_t=x,\forall t\geq 0)=1
\]
for any $x\in \left(\cup_{k\geq 1}\tI_k \right)^c$. Note incidentally that $\left(\cup_{k\geq 1}\tI_k \right)^c$ is not necessarily of zero $\fm$-measure. 

\subsection{Quasi notions}\label{SEC22}

A simple lemma below presents a typical nest for the Dirichlet form in \eqref{EQ2FSFL}. 

\begin{lemma}\label{LM21}
Let $(\EE^{(\fs_k)},\FF^{(\fs_k)})$ be the Dirichlet form on $L^2(\tI_k, \fm|_{\tI_k})$ given by \eqref{EQ2FSFL} with $(\tI_k,\fs_k)$ in place of $(\tI, \fs)$ and $\{F^k_m\subset \tI_k:m\geq 1\}$ an increasing sequence of closed intervals such that $\cup_{m\geq 1}F^k_m=\tI_k$. Then $\{F^k_m: m\geq 1\}$ is an $\EE^{(\fs_k)}$-nest.
\end{lemma} 
\begin{proof}
Note that $\FF^{(\fs_k)}\cap C_c(\tI_k)\subset \cup_{m\geq 1}\FF^{(\fs_k)}_{F^k_m}$. Thus we obtain the conclusion by the regularity of $(\EE^{(\fs_k)},\FF^{(\fs_k)})$. 
\end{proof}

In what follows, we shall characterize the quasi notions of $(\EE,\FF)$. These characterizations are valid not only for $(\EE,\FF)$ but also for the local Dirichlet forms appeared in \cite{LY172}.  
Since it may be of independent interest, we conclude it as a theorem.

\begin{theorem}\label{THM2}
Let $\fm$ and $(\EE,\FF)$ with the effective intervals $\{(\tI_k, \fs_k): k\geq 1\}$ be in Theorem~\ref{THM0}. Denote the $1$-capacity of $(\EE,\FF)$ by $\text{Cap}$. Then the following hold:
\begin{itemize}
\item[(1)] Let $\{F_m: m\geq 1\}$ be an $\EE$-nest. Then for any $k$, 
\[
	\{F^k_m:=F_m\cap \tI_k: m\geq 1\}
\]
is an $\EE^{(\fs_k)}$-nest. 
\item[(2)] Let $\{F^k_m\subset \tI_k:m\geq 1\}$ be an $\EE^{(\fs_k)}$-nest such that $F^k_m$ is closed in $\mathbb{R}$ for each $k\geq 1$ (such as that in Lemma~\ref{LM21}). Take another increasing sequence of closed sets $\{F^0_m:m\geq 1\}$ such that 
\[
	\fm\left(\left(\cup_{k\geq 1}\tI_k\right)^c\setminus (\cup_{m\geq 1}F^0_m)\right)=0.
\]
Write
\[
F_m:=\left(\cup_{k=1}^m F^k_m\right) \bigcup F^0_m.
\]
Then $\{F_m:m\geq 1\}$ is an $\EE$-nest. 
\item[(3)] Let $A\subset \mathbb{R}$. Then $A$ is $\EE$-polar, if and only if $A$ is contained in an $\fm$-negligible Borel subset of $(\cup_{k\geq 1}\tI_k)^c$. Particularly, every singleton of $\cup_{k\geq 1}\tI_k$ is of positive capacity. 
\item[(4)]  For any $k$ and any compact $K\subset \tI_k$, it holds that
\begin{equation}\label{EQ2XKX}
\inf_{x\in K} \text{Cap}(\{x\})>0.
\end{equation}
Particularly, if $\{F_m:m\geq 1\}$ is an $\EE$-nest, then for some $M\in \mathbb{N}$, 
\begin{equation}\label{EQ2KFM}
	K\subset F_m
\end{equation}
for all $m>M$. 
\item[(5)] Let $f$ be a measurable function on $\mathbb{R}$. Then $f$ is $\EE$-quasi-continuous, if and only $f|_{\tI_k}$ is continuous on $\tI_k$. 
\item[(6)] Let $\mu$ be a $\sigma$-finite positive measure on $\mathbb{R}$. Then $\mu$ is a smooth measure relative to $(\EE,\FF)$, if and only if $\mu\ll \fm$ on $\left(\cup_{k\geq 1}\tI_k\right)^c$ and $\mu|_{\tI_k}$ is a Radon measure on $\tI_k$ for any $k\geq 1$. 
\end{itemize}
\end{theorem}
\begin{proof}
\begin{itemize}
\item[(1)] Note that $\FF^{(\fs_k)}_{F^k_m}=\{f|_{\tI_k}: f\in \FF_{F_m}\}$ and $\FF^{(\fs_k)}=\{f|_{\tI_k}: f\in \FF\}$. Since $\cup_{m\geq 1}\FF_{F_m}$ is $\EE_1$-dense in $\FF$, we can deduce that $\cup_{m\geq 1}\FF^{(\fs_k)}_{F^k_m}$ is $\EE^{(\fs_k)}_1$-dense in $\FF^{(\fs_k)}$. This means $\{F^k_m:m\geq 1\}$ is an $\EE^{(\fs_k)}$-nest. 
\item[(2)] It suffices to show $\cup_{m\geq 1}\FF_{F_m}$ is $\EE_1$-dense in $\FF$. To this end, take $f\in \FF$ and fix an arbitrary small constant $\varepsilon>0$. Set $f_k:=f|_{\tI_k}$.  Since 
\[
	\sum_{k\geq 1}\EE_1^{(\fs_k)}(f_k,f_k)\leq \EE_1(f,f)<\infty,
\]
there exists an integer $K$ such that
\[
\sum_{k> K}\EE^{(\fs_k)}_1(f_k,f_k)<\varepsilon/3. 
\] 
For each $1\leq k\leq K$, it follows from $f_k\in \FF^{(\fs_k)}$ that for some integer $M_k$, there exists a function $g_k\in \FF^{(\fs_k)}_{F^k_{M_k}}$ such that
\[
	\EE_1^{(\fs_k)}(f_k-g_k, f_k-g_k)< \frac{\varepsilon}{3K}.
\]
On the other hand, some integer $M_0$ also exists such that
\[
	\int_{\left(\cup_{k\geq 1}\tI_k\right)^c\setminus F^0_{M_0}}f(x)^2\fm(dx)<\varepsilon/3. 
\]
Set $M:=M_0\vee M_1\vee \cdots \vee M_K \vee K$ and define a function $g$ by letting 
\[
\begin{aligned}
	&g:=g_k \text{ on }F^k_{M_k},\quad 1\leq k\leq K, \\
	&g:=f \text{ on }F^0_{M_0} \cap \left(\cup_{k\geq 1}\tI_k\right)^c,
\end{aligned}
\]
and otherwise $g:=0$. Clearly, for any $1\leq k\leq K$, $g|_{\tI_k}=g_k\in \FF^{(\fs_k)}_{F^k_{M_k}}\subset \FF^{(\fs_k)}_{F^k_M}$ since $M_k\leq M$. Thus $g\in \FF_{F_M}$ and
\[
\begin{aligned}
\EE_1&(g-f,g-f)\\ &\leq \sum_{1\leq k\leq K}\EE_1^{(\fs_k)}(f_k-g_k, f_k-g_k)+ \int_{\left(\cup_{k\geq 1}\tI_k\right)^c\setminus F^0_{M_0}}f(x)^2\fm(dx) \\ &\qquad + \sum_{k> K}\EE^{(\fs_k)}_1(f_k,f_k) \\
&<\varepsilon. 
\end{aligned}\]
Therefore, $\{F_m:m\geq 1\}$ is an $\EE$-nest. 
\item[(3)] From \cite[(2.2.40]{CF12}, we can conclude that every singleton of $\cup_{k\geq 1}\tI_k$ is of positive capacity. Thus every $\EE$-polar set is contained in an $\fm$-negligible Borel subset of $(\cup_{k\geq 1}\tI_k)^c$ by \cite[Theorem~4.1.1]{FOT11}. To the contrary, let $A\subset N$, where $N$ is a Borel subset of $(\cup_{k\geq 1}\tI_k)^c$ and $\fm(N)=0$. It suffices to show $N$ is $\EE$-polar. Note that $\text{Cap}$ is a Choquet capacity and any Borel set is capacitable. Hence without loss of generality, we may assume $N$ is compact. Then $N^c$ is open and can be written as a union of disjoint open intervals:
\begin{equation}\label{EQ2NCP}
	N^c=\cup_{p\geq 1}(c_p, d_p). 
\end{equation}
For each $k\geq 1$, take an $\EE^{(\fs_k)}$-nest $\{F^k_m: m\geq 1\}$ as in Lemma~\ref{LM21}. Further set
\begin{equation}\label{EQ2FMP}
F^0_m:=\cup_{p=1}^m [c_p+1/m, d_p-1/m].
\end{equation}
Then we have $F^0_m$ is increasing in $m$ and $\cup_{m\geq 1}F^0_m=N^c$. This indicates 
\[
	\fm\left((\cup_{k\geq 1}\tI_k)^c\setminus (\cup_{m\geq 1}F^0_m) \right)=\fm\left((\cup_{k\geq 1}\tI_k)^c\setminus N^c \right)=0.
\]
Let $F_m:=\left(\cup_{k=1}^m F^k_m\right) \bigcup F^0_m$. From the previous assertion, we know that $\{F_m:m\geq 1\}$ is an $\EE$-nest. We complete the proof with showing that $N\subset (\cup_{m\geq 1}F_m)^c$. Indeed,
\[
	 (\cup_{m\geq 1}F_m)^c= \cap_{m\geq 1} \left(\cap_{k=1}^m (F^k_m)^c \bigcap (F^0_m)^c \right). 
\]
For any $m\geq 1$ and $1\leq k\leq m$, since $F^k_m\subset \tI_k$, it follows that $N\subset (\cup_{k\geq 1}\tI_k)^c\subset (F^k_m)^c$. Moreover, $N\subset (F^0_m)^c$ is clear by \eqref{EQ2NCP} and \eqref{EQ2FMP}. Therefore, $N$ is $\EE$-polar. 
\item[(4)] Note that \eqref{EQ2XKX} is also indicated by \cite[(2.2.40)]{CF12}. Suppose \eqref{EQ2KFM} does not hold. Then $F_m^c\cap K\neq \emptyset$ for any $m\geq 1$. Thus 
\[
	\text{Cap}(K\setminus F_m)\geq   \inf_{x\in K} \text{Cap}(\{x\})>0,
\]
which contradicts with \cite[Theorem~1.3.14]{CF12}. 
\item[(5)] The necessity is obvious by \eqref{EQ2KFM}. For the sufficiency, let $f$ be such a function. For each $k\geq 1$, take an $\EE^{(\fs_k)}$-nest $\{F^k_m: m\geq 1\}$ as in Lemma~\ref{LM21}. Clearly, $f|_{F^k_m}$ is continuous. 
On the other hand, we assert that there exists an increasing sequence of closed sets $\{F^0_m: m\geq 1\}$ such that 
\[
	\fm\left((\cup_{m\geq 1}F^0_m)^c \right)=0
\]
and $f|_{F^0_m}$ is continuous on $F^0_m$ for any $m\geq 1$. In fact, for any $p\in \mathbb{Z}$, by Lusin's Theorem, we can take an increasing sequence of compact sets $\{K^p_m\subset [p,p+1]: m\geq 1\}$ such that $\fm([p,p+1]\setminus K^p_m)<1/m$ and $f|_{K^p_m}$ is continuous. Define
\[
	F^0_m:=\cup_{|p|\leq m}K^p_m,\quad m\geq 1. 
\]
Clearly, $\{F^0_m:m\geq 1\}$ is an increasing sequence of closed sets and $f|_{F^0_m}$ is continuous. Moreover, 
\[
	\cup_{m\geq 1} F^0_m=\cup_{m\geq 1}\cup_{|p|\leq m}K^p_m=\cup_{p\in \mathbb{Z}} \cup_{m\geq |p|}K^p_m
\]
and 
\[
W_p:=\cup_{m\geq |p|}K^p_m=\cup_{m\geq 1}K^p_m\subset [p,p+1]. 
\]
Hence 
\[
	\left(\cup_{m\geq 1} F^0_m\right)^c=(\cup_{p\in \bZ}[p,p+1])\setminus (\cup_{p\in \bZ}W_p)\subset \cup_{p\in \mathbb{Z}} ([p,p+1]\setminus W_p).
\]
This leads to 
\begin{equation}\label{EQ2MMFM}
\fm\left((\cup_{m\geq 1}F^0_m)^c \right)\leq \sum_{p\in \bZ} \fm([p,p+1]\setminus W_p).
\end{equation}
However, $\fm([p,p+1]\setminus W_p)\leq \inf_{m\geq 1}\fm\left( [p,p+1]\setminus K^p_m\right)=0$. We then obtain $\fm\left((\cup_{m\geq 1}F^0_m)^c \right)=0$.
Finally, set
\[
F_m=\left(\cup_{k=1}^m F^k_m\right) \bigcup F^0_m.
\]
By the second assertion, $\{F_m:m\geq 1\}$ is an $\EE$-nest. Since $f|_{F^k_m}$ and $f|_{F^0_m}$ are continuous, we can conclude that $f|_{F_m}$ is continuous. Therefore, $f$ is $\EE$-quasi-continuous. 
\item[(6)] The necessity is obvious by the third and fourth assertions. To prove the sufficiency, we need only to show there exists an $\EE$-nest $\{F_m:m\geq 1\}$ such that $\mu(F_m)<\infty$ for any $m\geq 1$. To this end, we take an $\EE^{(\fs_k)}$-nest $\{F^k_m:m\geq 1\}$ as in Lemma~\ref{LM21}. Since $F^k_m\subset \tI_k$ is a closed interval and $\mu|_{\tI_k}$ is Radon on $\tI_k$, we have $\mu(F^k_m)<\infty$. On the other hand, we assert that there exists an increasing sequence of closed sets $\{F^0_m: m\geq 1\}$ such that 
\begin{equation}\label{EQ2MMF}
	\fm\left((\cup_{m\geq 1}F^0_m)^c\right)=0, \text{ and } \mu(F^0_m)<\infty,\; \forall m\geq 1. 
\end{equation}
Indeed, denote $\nu:=\mu+\fm$. Since $\nu$ is $\sigma$-finite, we can take a sequence of sets $\{A_p: p\geq 1\}$ such that
\[
	\mathbb{R}=\cup_{p\geq 1}A_p, \quad \nu(A_p)<\infty, \quad p\geq 1. 
\]
Write the restriction of $\nu$ to $A_p$ by $\nu_p$, i.e. $\nu_p(\cdot):=\nu(\cdot \cap A_p)$. Then $\nu_p$ is a finite measure on $\mathbb{R}$ and hence Radon on $\bR$. It follows that there exists an increasing sequence of closed sets $\{K^p_m\subset A_p: m\geq 1\}$ such that $\nu(A_p\setminus K^p_m)=\nu_p(A_p\setminus K^p_m)<1/m$. Set
\[
	F^0_m:=\cup_{p=1}^m K^p_m. 
\]
Clearly, $\{F^0_m:m\geq 1\}$ is an increasing sequence of closed sets and
\[
	\mu(F^0_m)\leq \sum_{p=1}^m \mu(K^p_m)\leq \sum_{p=1}^m \nu(A_p)<\infty. 
\] 
Moreover, 
\[
	\cup_{m\geq 1}F^0_m=\cup_{m\geq 1}\cup_{p=1}^m K^p_m=\cup_{p\geq 1}\cup_{m\geq p} K^p_m=\cup_{p\geq 1}W_p,
\]
where $W_p:=\cup_{m\geq p}K^p_m$. Mimicking \eqref{EQ2MMFM}, we can obtain
\[
\fm\left((\cup_{m\geq 1}F^0_m)^c\right)=\sum_{p\geq 1}\fm\left(A_p\setminus W_p \right)\leq \sum_{p\geq 1}\nu(A_p\setminus W_p)=0.
\]
This leads to \eqref{EQ2MMF}. Finally, let
\[
F_m=\left(\cup_{k=1}^m F^k_m\right) \bigcup F^0_m.
\]
Then $\{F_m:m\geq 1\}$ is an $\EE$-nest and $\mu(F_m)<\infty$. Therefore, $\mu$ is a smooth measure. 
\end{itemize} 
That completes the proof.
\end{proof}

We present a corollary to characterize the quasi notions of $(\EE,\FF)$ for the situation where $\left(\cup_{k\geq 1}\tI_k\right)^c$ is of zero $\fm$-measure. This is satisfied in \eqref{EQ2DEC}, since there we have
\[
\fm\left((\cup_{k\geq 1}\tI_k)^c \right)= \int_{(\cup_{k\geq 1}\tI_k)^c}\rho(x)dx=\int_{S(\rho)}\rho(x)dx=0. 
\]

\begin{corollary}\label{COR1}
Let $(\EE,\FF)$ be the Dirichlet form in Theorem~\ref{THM2} and assume further $\fm\left((\cup_{k\geq 1}\tI_k)^c \right)=0$. Then the following hold:
\begin{itemize}
\item[(1)] Let $\{F^k_m\subset \tI_k:m\geq 1\}$ be an $\EE^{(\fs_k)}$-nest such that $F^k_m$ is closed in $\mathbb{R}$ for each $k\geq 1$ and set $F_m:=\cup_{k=1}^mF^k_m$. Then $\{F_m:m\geq 1\}$ is an $\EE$-nest. 
\item[(2)] $\left(\cup_{k\geq 1}\tI_k\right)^c$ is $\EE$-polar and every $\EE$-polar set is a subset of $\left(\cup_{k\geq 1}\tI_k\right)^c$. 
\item[(3)] Let $\mu$ be a positive measure on $\mathbb{R}$. Then $\mu$ is smooth relative to $(\EE,\FF)$, if and only if $\mu(\left(\cup_{k\geq 1}\tI_k\right)^c)=0$ and $\mu|_{\tI_k}$ is Radon on $\tI_k$ for any $k\geq 1$. 
\end{itemize}
\end{corollary}

\subsection{Conservativeness}\label{SEC23}

We present a condition:
\begin{itemize}
\item[(C)] $X$ is conservative. In other words, 
\[
	\mathbf{P}_x(\zeta=\infty)=1
\]
for any $x\in \bR$, where $\zeta$ is the lifetime of $X$.
\end{itemize}
The following result concludes a characterization of the conservativeness.

\begin{proposition}\label{THM4}
Let $(\EE,\FF)$ and $X$ be in Theorem~\ref{THM0}.
Then (C) does not hold, if and only if either of the following holds: 
\begin{itemize}
\item[(1)] $\fa_k=-\infty$ for some $k$, i.e. $\tI_k=(-\infty, \fb_k\rangle$, and 
\begin{equation}\label{EQ4EKM}
	\int_{-\infty}^{\fe_k}\fm\left( (x,\fe_k)\right)\lambda_{\fs_k}(dx)<\infty;
\end{equation}
\item[(2)] $\fb_k=\infty$ for some $k$, i.e. $\tI_k=\langle\fa_k, \infty)$, and 
\begin{equation}\label{EQ4EKME}
	\int^{\infty}_{\fe_k}\fm\left( (\fe_k,x)\right)\lambda_{\fs_k}(dx)<\infty.
\end{equation}
\end{itemize}
Particularly, if all intervals in $\{\tI_k:k\geq 1\}$ are bounded, then $X$ is conservative. 
\end{proposition} 
\begin{proof}
Firstly, we note that either of these two conditions above implies that the restriction of $X$ to $\tI_k$ is not conservative by \cite[(3.5.13)]{CF12}. Thus $X$ is not conservative. On the contrary, the non-conservativeness of $X$ leads to that of the restriction of $X$ to some $\tI_k$. Since the restriction of $X$ to any bounded $\tI_k$ is recurrent (and hence conservative) by \cite[Theorem~2.2.11]{CF12}, it follows that $\fa_k= -\infty$ or $\fb_k= \infty$. Finally, we can conclude \eqref{EQ4EKM} or \eqref{EQ4EKME} by using \cite[(3.5.13)]{CF12} again. That completes the proof. 
\end{proof}


\bibliographystyle{abbrv}
\bibliography{GSBM}

\end{document}